\documentclass[11pt,a4paper]{amsart}
\usepackage{amssymb}
\usepackage{latexsym}
\usepackage{exscale}
\usepackage{amsfonts}
\usepackage{graphicx}
\usepackage{mathrsfs}
\usepackage{amsmath,amscd,amsthm}
\usepackage{bbm,pifont}
\usepackage{enumerate}
\usepackage{color}
\usepackage{amssymb}
\usepackage{latexsym}
\usepackage{exscale}

\allowdisplaybreaks

\numberwithin{equation}{section}
\newtheorem{theorem}{Theorem}[section]
\newtheorem{lemma}[theorem]{Lemma}
\newtheorem{corollary}[theorem]{Corollary}
\newtheorem{proposition}[theorem]{Proposition}

\newtheorem{definition}[theorem]{Definition}

\newcommand{\supp}{\operatorname{supp}}

\def\supp{\mathop{\rm supp}}

\def\X{{\mathrm{X}}}

\def\H{\mathbb H}

\begin{document}
\allowdisplaybreaks

\title[Riesz transform and Commutator Theorems on stratified Lie groups]
{ Lower bound of Riesz transform kernels revisited and commutators on stratified Lie groups}
\author{Xuan Thinh Duong, Hong-Quan Li, Ji Li, Brett D. Wick and Qingyan Wu}

\address{Xuan Thinh Duong, Department of Mathematics, Macquarie University, NSW, 2109, Australia}
\email{xuan.duong@mq.edu.au}

\address{Hong-Quan Li, School of Mathematical Sciences, Fudan University, 220 Handan Road, Shanghai 200433, People's Republic of China}
\email{hongquan\_li@fudan.edu.cn}

\address{Ji Li, Department of Mathematics, Macquarie University, NSW, 2109, Australia}
\email{ji.li@mq.edu.au}

\address{Brett D. Wick, Department of Mathematics\\
         Washington University - St. Louis\\
         St. Louis, MO 63130-4899 USA
         }
\email{wick@math.wustl.edu}

\address{Qingyan Wu, Department of Mathematics\\
         Linyi University\\
         Shandong, 276005, China
         }
\email{wuqingyan@lyu.edu.cn}

\date{\today}

\begin{abstract}
Let $\mathcal G$ be a stratified  Lie group and $\{\X_j\}_{1 \leq j \leq n}$ a basis for the left-invariant vector fields of degree one on $\mathcal G$.
Let $\Delta = \sum_{j = 1}^n \X_j^2 $ be the sub-Laplacian
on $\mathcal G$ and the $j^{\mathrm{th}}$  Riesz transform on $\mathcal G$  is defined by $R_j:= \X_j (-\Delta)^{-\frac{1}{2}}$,
 $1 \leq j \leq n$. 
In this paper we give a new version of the lower bound of the kernels of Riesz transform $R_j$ and then establish the Bloom-type two weight estimates as well as
a number of endpoint characterisations  for the commutators of the Riesz transforms and BMO functions, including the $L\log^+L(\mathcal G)$ to weak $L^1(\mathcal G)$, $H^1(\mathcal G)$ to $L^1(\mathcal G)$ and $L^\infty(\mathcal G)$ to BMO$(\mathcal G)$.
Moreover, we also study the behaviour of the Riesz transform kernel on a special case of stratified  Lie group: the Heisenberg group, and then we obtain the weak type $(1,1)$ characterisations for the Riesz commutators.
\end{abstract}

\maketitle


\section{Introduction and statement of main results}
\setcounter{equation}{0}

The Calder\'on-Zygmund theory of singular integrals has a central role in modern harmonic analysis with 
extensive applications to other fields such as partial differential equations and complex analysis. A prototype of 
singular integral on the real line is the Hilbert transform which is bounded on $L^p(\mathbb R)$ for $1 < p < \infty$.
At the end-point $p =1$, the Hilbert transform is bounded from $L^1(\mathbb R)$ to $L^{1,\infty}(\mathbb R)$ and 
is bounded from the Hardy space 
$H^1(\mathbb R)$ into $L^1(\mathbb R)$ while at the end-point $p = \infty$, the Hilbert transform is 
bounded from $L^{\infty} (\mathbb R)$ to the BMO space BMO$(\mathbb R)$.
Further study on singular integrals and the related partial differential equations leads to the commutator $[b, H]$ of the Hilbert transform $H$ and a BMO function $b$
defined by 
$$[b, H] f(x)  = b(x) Hf(x) - H(bf)(x) $$
for suitable functions $f$ (introduced by A.P. Calder\'on \cite{C}). It is well-known that the commutator of the Hilbert transform has the following 
properties:


a) $ \|[b,H]\|_{L^p(\mathbb R)\to L^p(\mathbb R)} \approx \|b\|_{{\rm BMO}(\mathbb R)}$ (\cite{CRW}); 

b) $ \|[b,H]\|_{L_\mu^p(\mathbb R)\to L^p_\lambda(\mathbb R)} \approx \|b\|_{{\rm BMO}_\nu(\mathbb R)}$ with $\mu,\lambda\in A_p$ for $1<p<\infty$ and $\nu=\big({\mu\over\lambda}\big)^{1\over p}$, where $A_p$ denotes the Muckenhoupt weights (\cite{B},\cite{HLW}); 

c) $[b,H]$ is bounded from $L\log^+L(\mathbb R)$ to $L^{1,\infty}(\mathbb R)$ if and only if $b\in {\rm BMO}(\mathbb R)$  (\cite{P},\cite{GLW},\cite{Ac});

d) $[b,H]$ is bounded from $H^1(\mathbb R)$ to $L^{1}(\mathbb R)$ if and only if $b$ equals a constant almost everywhere (\cite{HST});

e) $[b,H]$ is bounded from $L_c^\infty(\mathbb R)$ to ${\rm BMO}(\mathbb R)$ if and only if $b$ equals a constant almost everywhere (\cite{HST});

f) $[b,H]$ is of weak type $(1,1)$ if and only if $b$ is in $L^\infty(\mathbb R)$ (\cite{Ac}).

For more details, we refer to the references listed above. We also point out that there are quite a number of recent results on 
the characterisations of commutators in the above forms for singular integrals in different settings, see for example
\cite{DLWY,LW,DLLW,LNWW,TYY,P,HLW,Ler, GLW, DLHWY,Ac}.

%

Inspired by these classical results above, it is natural to ask
whether these results hold on the Heisenberg group $\mathbb H^n$. Note that in several complex variables,
the Heisenberg group $\mathbb H^n$ is the boundary of the Siegel upper half space, whose roles are holomorphically equivalent to the unit sphere and the unit ball in $\mathbb C^n$. And hence, the role of the Riesz transform on 
$\mathbb H^n$ is similar to the role of Hilbert transform on the real line.

The first four authors in \cite{DLLW} established an analogous results of a) above in the more general setting of stratified   Lie groups 
by studying the behaviour of the Riesz transform kernels and obtaining  a lower bound on the kernel. Note that  the Riesz transform kernel does not have an explicit representation, 
and hence the key difficulty is to obtain a suitable version of kernel lower bounds. To overcome this, in \cite{DLLW} they studied and made good use of the group structures
and the dilations related to the stratified condition, and then established a first version of the kernel lower bound for Riesz transforms.

However, it is not clear whether the kernel lower bound for Riesz transforms introduced in \cite{DLLW} is enough to further study 
analogous results of b), c), d) and e) on stratified   Lie groups. 

The aim of this paper is three fold. First, we provide a better understanding of the Riesz transform kernel behaviour on stratified   Lie groups, which is stronger than the lower bound obtained in \cite{DLLW} and is new in the literature. Second, we apply these kernel lower bounds to study the endpoint characterisations of boundedness of commutators, i.e., we establish
similar versions of b), c), d) and e) for the Riesz commutators (but with some improvement) on stratified   Lie groups.
Third, we also study the characterisation of boundedness from $L^1$ to $L^{1,\infty}$. However, we establish 
this result only on Heisenberg groups since its proof requires the result that the kernel of the Riesz transform  on Heisenberg group
is non-zero almost everywhere, which we will prove in this paper. It is still an open question for this result on general stratified Lie groups. 

To be more specific, suppose $\mathcal G$ is a stratified   Lie group. Let $\{\X_j\}_{1 \leq j \leq n}$
be a basis for the left-invariant vector fields of degree one on $\mathcal G$.
Let $\Delta = \sum_{j = 1}^n \X_j^2 $ be the sub-Laplacian
on $\mathcal G$. Consider the $j^{\mathrm{th}}$  Riesz transform on $\mathcal G$ which is defined as $R_j:= \X_j (-\Delta)^{-\frac{1}{2}}$.

It is well-known that the Riesz transform $R_j$ is a Calder\'on--Zygmund operator on $\mathcal G$, i.e.,
it is  bounded on $L^2(\mathcal G)$ and the kernel satisfies the corresponding size and smoothness conditions, see for example \cite{CG,FoSt}. Moreover,
it is also well-known that from \cite{CG}, the set comprising the identity operator together with Riesz transforms, $\{I, R_1, R_2, \ldots, R_n\}$ characterises the Hardy space
$H^1(\mathcal G)$ introduced and studied by Folland and Stein \cite{FoSt}.

We now recall the BMO space on $\mathcal G$, which is the dual space of $H^1(\mathcal G)$ \cite[Chapter 5]{FoSt}, defined as
$$ {\rm BMO}(\mathcal G):=\{ b\in L^1_{loc}(\mathcal G):\  \|b\|_{{\rm BMO}(\mathcal G)}<\infty \},$$
where
\begin{align}\label{BMO norm}
\|b\|_{{\rm BMO}(\mathcal G)}:=\sup_B {1\over |B|}\int_B|b(g)-b_B|dg,
\end{align}
and $b_B:={1\over |B|}\int_Bb(g)\,dg$, where $B$ denotes the ball on $\mathcal G$ defined via a homogeneous norm $\rho$. We also mention that the weighted BMO space ${\rm BMO}_\nu(\mathcal G)$ is defined via replacing the measure $|B|$ in \eqref{BMO norm} by $\nu(B)$.  See Section 2 for details.

The first part of main results of this paper is to establish an enhanced version of the lower bound of the Riesz transform kernel on stratified Lie groups and then study the properties of commutators, including the following:

${\rm b})'$ $ \|[b,R_j]\|_{L_\mu^p(\mathcal G)\to L^p_\lambda(\mathcal G)} \approx \|b\|_{{\rm BMO}_\nu(\mathcal G)}$ with $\mu,\lambda\in A_p(\mathcal G)$ for $1<p<\infty$ and $\nu=\big({\mu\over\lambda}\big)^{1\over p}$;

${\rm c})'$ $[b,R_j]$ is bounded from $L\log^+L(\mathcal G)$ to $L^{1,\infty}(\mathcal G)$ if and only if $b\in {\rm BMO}(\mathcal G)$;


${\rm d})'$ $[b,R_j]$ is bounded from $H^1(\mathcal G)$ to $L^{1}(\mathcal G)$ if and only if $b$ equals a constant almost everywhere;

${\rm e})'$ $[b,R_j]$ is bounded from $L_c^\infty(\mathcal G)$ to ${\rm BMO}(\mathcal G)$ if and only if $b$ equals a constant almost everywhere.

\smallskip
The second part of the main results of this paper is to consider the characterisation of
 the endpoint boundedness of the commutator with respect to the weak type $(1,1)$, i.e., we aim to study:
 
$ f)'$ $[b,R_j]$ is of weak type $(1,1)$ if and only if $b$ is in $L^\infty(\mathcal G)$.

\smallskip 
However, it is not clear whether one can establish this result on general stratified Lie groups $\mathcal G$. Here 
 we study the characterisation of boundedness of commutator from $L^1$ to $L^{1,\infty}$ by focusing on the special case of stratified Lie group: the Heisenberg group $\mathbb H^n$, which can be identified with $\mathbb C^n\times\mathbb R$ with the group structure (for all the notation on $\mathbb H^n$ we refer to Section 2).


To be more specific, we begin by recalling
a recent result by the first four authors \cite{DLLW}, where they studied the behaviour of the kernel $K_j$ of Riesz transform $R_j$ on $\mathcal G$ and obtained that: $K_j \not\equiv 0$ in $\mathcal{G} \setminus \{ 0 \}$.
Then, based on this behaviour they further obtained 
 the following version of kernel bounds which implies an analogues of a) for the commutator of $R_j$ and the BMO space on $\mathcal G$.

\smallskip
\noindent {\bf Theorem A\ }(\cite{DLLW}){\bf.} {\it
Fix $j=1,\ldots,n$. There exist $ 0<\varepsilon_o\ll1$ and $C>0$ such that  for any $0 <\eta< \varepsilon_o$ and for all  $g \in \mathcal G$ and $r > 0$, we can find $g_* = g_*(j, g,r) \in \mathcal G$ satisfying
\begin{align}
\rho(g, g_*) = r, \quad |K_j(g_1, g_2)| \geq C r^{- Q}, \quad \forall g_1 \in B(g, \eta r), \  g_2 \in B(g_*, \eta r).
\end{align}
}
%
  %
%
To obtain ${\rm b})'$, ${\rm c})'$, ${\rm d})'$ and ${\rm e})'$,  we point out that the result 
in Theorem A 
may not be enough. To achieve this, we need an 
enhanced version of the kernel lower bound. Thus,
the main results of this paper are three fold. First, we establish an enhanced version of the kernel lower bound of the Riesz transform
 $R_j$ on $\mathcal G$ via constructing a type of ``twisted truncated sector'' on $\mathcal G$, which is of independent interest and will be useful in studying other problems.
 \begin{theorem}\label{thm0}
Suppose that $\mathcal G$ is a stratified Lie group with homogeneous dimension $Q$ and that $j\in\{1,2,\ldots,n\}$.  
There exist a large positive constant $r_o$ and a positive constant $C$  such that for every 
$g\in \mathcal G$ there exists a ``twisted truncated sector'' $G\subset \mathcal G$ such that $\inf_{g'\in G}\rho(g,g')=r_o$ and that  for every $g_1\in B(g,1)$ and $g_2\in G$, we have
\begin{align}\label{lower bound}
 |K_j(g_1,g_2)|\geq C  \rho(g_1,g_2)^{-Q}, \quad |K_j(g_2,g_1)|\geq C  \rho(g_1,g_2)^{-Q},
\end{align}
and all $K_j(g_1,g_2)$ as well as all $K_j(g_2,g_1)$ have the same sign.

Moreover, this ``twisted truncated sector''  $G$ is regular, in the sense that  $|G|=\infty$ and that for any $R>2r_o$,
\begin{align}\label{regular}
 |B(g,R) \cap G|\approx R^Q,  
\end{align}
where the implicit constants are independent of $g$ and $R$. 

\end{theorem}

Here we point out that the set $G$ that we constructed in Theorem \ref{thm0} above 
is a connected open set spreading out to infinity, which plays the role of 
the ``truncated sector centered at a fixed point'' in the Euclidean setting. 
The shape of $G$ here may not be the same as the usual sector since the norm $\rho$ (or the Carnot--Carath\'eodory metric $d$) on $\mathcal G$ is different from the standard Euclidean metric. However, such a kind of twisted sector always exists.



Second, we establish the Bloom-type two weight estimates for the commutators $[b,R_j]$. 
\begin{theorem}\label{thm two weight}
Let $\mu,\lambda\in A_p(\mathcal G)$, $1<p<\infty$. Further set $\nu= \big( {\mu\over\lambda}\big)^{1\over p}$. Suppose $j\in\{1,\ldots,n\}$. Then

{\rm (i)} if $b\in {\rm BMO}_\nu(\mathcal G)$, then 
$$ \|[b,R_j](f)\|_{L^p_\lambda(\mathcal G)} \lesssim \|b\|_{  {\rm BMO}_\nu(\mathcal G) } ([\lambda]_{A_p(\mathcal G)} \cdot [\mu]_{A_p(\mathcal G)} )^{\max\{1, {1\over p-1}\}}\|f\|_{L^p_\mu(\mathcal G)}.$$

{\rm (ii)} for every $b\in L^1_{loc}(\mathcal G)$, if 
$[b,R_j]$ is bounded from $L^p_\mu(\mathcal G)$ to $L^p_\lambda(\mathcal G)$, then
$b\in {\rm BMO}_\nu(\mathcal G)$ with 
$$     \|b\|_{  {\rm BMO}_{\nu}(\mathcal G) } \lesssim   \|[b,R_j]\|_{L^p_\mu(\mathcal G)\to L^p_\lambda(\mathcal G)}. $$

\end{theorem}
For the proof of part (i) in the above theorem, we point out that it follows directly from \cite{Ler} (see also \cite{HLW}, where they neglected the sharp constant argument) with only minor changes since for this part we only need to use the upper bound of the Riesz transform, which satisfies the standard size and smoothness condition of Calder\'on--Zygmund type. For part (ii), we use the idea and techniques originated from \cite{Ler} and then adapt it to our setting according to the lower bound of Riesz kernel obtained in Theorem \ref{thm0}.  


Next, we establish the endpoint estimates of the commutators $[b,R_j]$. 
\begin{theorem}\label{thm1}
Suppose that $\mathcal G$ is a stratified Lie group and that $j=1,2,\ldots,n$. Then 
given a function $b\in L^1_{loc}(\mathcal G)$,
$b\in {\rm BMO}(\mathcal G)$ 
if and only if $[b,R_j]$ is bounded from $L\log^+L(\mathcal G)$ to weak $L^1(\mathcal G)$, i.e., there exists a constant $\theta_b$ such that
for all $\lambda>0$ and for all $f\in L\log^+L(\mathcal G)$,
  \begin{align}  
 \big | \{ g\in\mathcal G:\ |[b,R_j](f)(g)|>\lambda\} \big| \leq \theta_b \int_{\mathcal G} {|f(g)|\over\lambda} \bigg(1+ \log^+ \Big( {|f(g)|\over\lambda} \Big)\bigg)dg.
  \end{align}
\end{theorem}
The proof of this theorem is twofold, the necessity and sufficiency.  For the sufficiency, we point out that it follows directly from \cite{P} with only minor changes since for this part we only need to use the upper bound of the Riesz transform, which satisfies the standard size and smoothness condition of Calder\'on--Zygmund type. For the necessary part, we use the idea of Uchiyama \cite{U} and the technique that has been further explored and studied in \cite{GLW} and \cite{Ac}. To be more specific, we 
write $|[b,R_j](f)(g)|\geq |R_j(bf)(g)| - |b(g)|\,|R_j(f)(g)|  $, and then by choosing a suitable function $f$ that is closely related to $b$ and with cancellation condition and by making good use of the lower bound in Theorem \ref{thm0}, we show that $|R_j(bf)(g)| $ is the main term and $|b(g)|\,|R_j(f)(g)|$ acts as the  ``error'' term due to the cancellation of $f$ and the smoothness condition of the kernel of $R_j$. Together with the boundedness of $[b,R_j]$, we show that $b$ is in BMO$(\mathcal G)$.

We also point out that this endpoint characterisation in Theorem \ref{thm1} above  is sharp since following the method in \cite{P} in the Euclidean setting and using the lower bound in Theorem \ref{thm0}, it is easy to construct a function $b\in {\rm BMO}(\mathcal G)$ such that $[b,R_j]$ fails to be weak type $(1,1)$.



Third, besides the weak type $(1,1)$, it is natural to study the endpoint estimates like $H^1(\mathcal G)$ to $L^1(\mathcal G)$ and
the $L^\infty(\mathcal G)$ to ${\rm BMO}(\mathcal G)$.  
We also have the following characterisations. Denote by $L_c^\infty(\mathcal G)$ the subspace of $L^\infty(\mathcal G)$ of compactly supported functions.
\begin{theorem}\label{thm2}
Suppose that $\mathcal G$ is a stratified Lie group, $b\in {\rm BMO}(\mathcal G)$  and $j\in\{1,2,\ldots,n\}$. Then 
the following statements are equivalent:
\begin{itemize}
\item[(i)] $[b,R_j]$ is bounded from $H^1(\mathcal G)$ to $L^1(\mathcal G)$;
\item[(ii)] $[b,R_j]$ is bounded from ${L_c^\infty(\mathcal G)}$ to ${\rm BMO}(\mathcal G)$;
\item[(iii)] $b$ equals a constant almost everywhere.
 \end{itemize}
\end{theorem}
The proof of this theorem follows from the strategy of the classical results \cite{HST}. We make use of the Riesz transform kernel lower bound in Theorem \ref{thm0} and adapt the idea in \cite{HST} to 
this setting  with necessary changes.

Next, we turn to the study of the behaviour of
the kernel of the Riesz transform on $\mathbb H^n$. 
\begin{theorem}\label{thm H lower bound}
Suppose $j=1,2,\ldots,2n$, and $R_j := X_j (-\Delta_{\mathbb H^n})^{-{1\over2}}$,  is the $j$th Riesz transform on $\mathbb H^n$. Let $K_j(g)$, $g\in\mathbb H^n$, be the kernel of $R_j$.  Then we have
$$K_j(g)\not=0,\quad {\rm\ a.e.\ } g\in\mathbb H^n.$$
\end{theorem}
This is an obvious fact for classical Hilbert transforms and Riesz transforms. However, it is not known before for Riesz transforms on Heisenberg groups. 
This is still open on general stratified Lie groups.

Next, concerning the commutator $[b,R_j]$ to be of weak type $(1,1)$, we have the following characterisation.
\begin{theorem}\label{thm3}
Suppose $j=1,2,\ldots,2n$, and $R_j := X_j (-\Delta_{\mathbb H^n})^{-{1\over2}}$,  is the $j$th Riesz transform on $\mathbb H^n$. Then 
given a function $b\in L^1_{loc}(\mathbb H^n)$, $[b,R_j]$ is of weak type $(1,1)$
if and only if $b\in L^\infty(\mathbb H^n)$.
\end{theorem}
The proof of this theorem follows from the result in Theorem \ref{thm H lower bound} and from the idea and approach in \cite{Ac} for Hilbert transform.

This paper is organised as follows. In Section 2 we recall the necessary preliminaries on stratified   Lie groups $\mathcal G$. In Section 3 we give a deep study of the behaviour of the Riesz transform kernels and the kernel lower bounds, and then prove Theorem \ref{thm0}. 
In Section 4, by using the kernel lower bound that
we established and the original idea in \cite{Ler}, we prove Theorem \ref{thm two weight}. In Section 5, by using the kernel lower bound that
we establish, we prove Theorem \ref{thm1}, the characterisation of BMO$(\mathcal G)$ via the endpoint ($L\log^+L$ to weak $L^1$) estimates of 
$[b,R_j]$. 
In Section 6 we  prove Theorem \ref{thm2} and 
in the last section we show the behaviour of the Riesz transform kernel on $\mathbb H^n$ (Theorem \ref{thm H lower bound}) and then give the characterisation of boundedness of Riesz commutator from $L^1$ to $L^{1,\infty}$ (Theorem \ref{thm3}).




\section{Preliminaries on stratified   Lie groups $\mathcal G$}
\setcounter{equation}{0}

Recall that a connected, simply connected nilpotent Lie group $\mathcal{G}$ is said to be stratified if its left-invariant Lie algebra $\mathfrak{g}$ (assumed real and of finite dimension) admits a direct sum decomposition
\begin{align*}
\mathfrak{g} = \bigoplus_{i = 1}^k V_i , \quad [V_1, V_i] = V_{i + 1}, \  \mbox{ for $i \leq k - 1$ and $[V_1, V_k]=0$;}
\end{align*}
$k$ is called the step of the group $\mathcal G$.

One identifies $\mathfrak{g}$ and $\mathcal{G}$ via the exponential map
\begin{align*}
\exp: \mathfrak{g} \longrightarrow \mathcal{G},
\end{align*}
which is a diffeomorphism.


We fix once and for all a (bi-invariant) Haar measure $dx$ on $\mathcal G$ (which is just the lift of Lebesgue measure on $\frak g$ via $\exp$).

There is a natural family of dilations on $\frak g$ defined for $r > 0$ as follows:
\begin{align*}
\delta_r \bigg( \sum_{i = 1}^k v_i \bigg) = \sum_{i = 1}^k r^i v_i, \quad \mbox{with $v_i \in V_i$}.
\end{align*}

We choose once and for all a basis $\{\X_1, \cdots, \X_n\}$ of $V_1$ and consider the sub-Laplacian $\Delta = \sum_{j=1 }^n \X_j^2 $. Observe that $\X_j$ ($1 \leq j \leq n$) is homogeneous of degree $1$ with respect to the dilations, and $\Delta$ of degree $2$ in the sense that:
\begin{align*}
&\X_j \left( f \circ \delta_r \right) = r \, \left( \X_j f \right) \circ \delta_r, \qquad  1 \leq j \leq  n, \  r > 0, \ f \in C^1, \\
&\delta_{\frac{1}{r}} \circ \Delta \circ \delta_r = r^2 \, \Delta, \quad \forall r > 0.
\end{align*}

For $i=1,\cdots,k$, let $n_i=\dim V_i$ and $m_i= n_1+\cdots+n_i$ and $m_0=0$, clearly, $n_1=n$. Set $N=m_k$.
Two important families of diffeomorphisms of $\mathcal G$ are the translations and dilations of $\mathcal G$. For any $g\in \mathcal G$, the (left) translation $\tau_g : \mathcal{G}\rightarrow \mathcal{G}$ is defined as
$$\tau_g (g')=g\circ g'.$$
The dilations on $\mathfrak g$ allow the definition of dilation on $\mathcal{G}$, which we still denote by $\delta_r$.
 For any $\lambda>0$, the dilation $\delta_{\lambda}:\mathcal{G}\rightarrow \mathcal{G}$, is defined as
\begin{equation}\label{delta}
 \delta_{\lambda}(g)=\delta_{\lambda}\left(x_1,x_2,\cdots,x_N\right)=\left(\lambda^{\alpha_1} x_1,\lambda^{\alpha_2} x_2,\cdots,\lambda^{\alpha_N}x_N\right),
\end{equation}
where $\alpha_j=i$ whenever $m_{i-1}<j\leq m_i$, $i=1,\cdots,k$. Therefore, $1=\alpha_1=\cdots=\alpha_{n_1}<\alpha_{n_1+1}=2\leq\cdots\leq \alpha_n=k.$ For any set $E\subset \mathcal G$, denote by $\tau_g(E)=\{g\circ g': g'\in E\}$ and $\delta_{r}(E)=\{\delta_r(g): g\in E\}$.

Let $Q$ denote the homogeneous dimension of $\mathcal{G}$, namely,
\begin{align}\label{homo dimension}
Q= \sum_{i=1}^k i\, {\rm dim} V_i.
\end{align}
And let $p_h$ ($h > 0$) denote the heat kernel (that is, the integral kernel of $e^{h \Delta}$)
on $\mathcal G$. For convenience, we set $p_h(g) = p_h(g, 0)$ (that is, in this note, for a convolution operator, we will identify the integral kernel with the convolution kernel) and $p(g) = p_1(g)$.

Recall that (c.f. for example \cite{FoSt})
\begin{align} \label{hkp1}
p_h(g) = h^{-\frac{Q}{2}} p(\delta_{\frac{1}{\sqrt{h}}}(g)), \qquad  \forall h > 0, \  g \in \mathcal G.
\end{align}

The kernel of the $j^{\mathrm{th}}$  Riesz transform $\X_j (-\Delta)^{-\frac{1}{2}}$ ($1 \leq j \leq  n$) is written simply as $K_j(g, g') = K_j(g'^{-1} \circ g)$. It is well-known that
\begin{align} \label{kjs}
K_j \in C^{\infty}(\mathcal G \setminus \{0\}), \ K_j(\delta_r(g)) = r^{-Q} K_j(g), \quad \forall g \neq 0, \ r > 0, \ 1 \leq j \leq  n,
\end{align}
which also can be explained by \eqref{hkp1} and the fact that
\begin{align*}
K_j(g) = \frac{1}{\sqrt{\pi}} \int_0^{+\infty} h^{-\frac{1}{2}} \X_j p_h(g) \, dh  = \frac{1}{\sqrt{\pi}} \int_0^{+\infty} h^{- \frac{Q}{2} - 1} \left( \X_j p \right)(\delta_{\frac{1}{\sqrt{h}}}(g)) \, dh.
\end{align*}

Next we recall the homogeneous norm $\rho$ (see for example \cite{FoSt}) on $\mathcal G$
which is defined to be a continuous function
$g\to \rho(g)$ from $\mathcal G$ to $[0,\infty)$, which is $C^\infty$ on $\mathcal G\backslash \{0\}$
%
and satisfies
\begin{enumerate}
\item[(a)] $\rho(g^{-1}) =\rho(g)$;
\item[(b)] $\rho({ \delta_r(g)}) =r\rho(g)$ for all $g\in \mathcal G$ and $r>0$;
\item[(c)] $\rho(g) =0$ if and only if $g=0$.
\end{enumerate}
For the existence (also the construction) of the homogeneous norm $\rho$ on  $\mathcal G$,
we refer to  \cite[Chapter 1, Section A]{FoSt}. { For convenience,
we set
\begin{align*}
\rho(g, g') = \rho(g'^{-1} \circ g) = \rho(g^{-1} \circ g'), \quad \forall g, g' \in \mathcal{G}.
\end{align*}
Recall that  this defines a quasi-distance in the sense of  Coifman--Weiss, in fact, we have the following improved pseudo-triangle inequality, there exists a constant $C_\rho \geq1$ such that (see \cite{BLU})
 \begin{align} \label{qdr}
|\rho(g_1, g_2) -\rho(g_1, g_3) |\leq C_\rho \rho(g_2, g_3) \qquad \forall g_1, g_2, g_3 \in  \mathcal{G}.
\end{align}


We now denote by $d$ the Carnot--Carath\'eodory metric associated to $\{\X_j\}_{1\leq j\leq n}$, which is equivalent to $\rho$ in the sense that there exists constants $C_1,C_2>0$ such that
for every $g_1,g_2\in\mathcal G$ (see \cite{BLU}),
\begin{align} \label{equi metric d}
 C_1\rho(g_1,g_2)\leq d(g_1,g_2)\leq C_2 \rho(g_1,g_2).
\end{align}
We point out that {the Carnot--Carath\'eodory metric $d$ even on the most special stratified Lie group, the Heisenberg group, is not smooth on $\mathcal G \setminus \{0\}$.}

In the sequel, to avoid confusing notation, for $g\in\mathcal G$ and $r>0$, {$B(g,r)$  
denotes the open ball 
defined by $\rho$. 

For the Folland--Stein BMO space  ${\rm BMO}(\mathcal G)$,  note that we have an {equivalent norm}, defined by
 $$\|b\|'_{\operatorname{BMO}(G)}=\sup_{B}\inf_{c}\frac{1}{|B|}\int_{B}|b(x)-c|dx.$$
 {For a ball B}, the infimum above is attained and the constants where this happens can be found among the median values.
 
 \begin{definition}\label{def median}
 A {median value} of a function $b$ over a ball $B$ will be any real number $m_b(B)$ that satisfies simultaneously
 $$\left|\{x\in B: b(x)<m_b(B)\}\right|\geq \frac{1}{2}|B|$$
 and
 $$\left|\{x\in B: b(x)>m_b(B)\}\right|\geq \frac{1}{2}|B|.$$
 \end{definition}
Following the standard proof in \cite[p.199]{To}, we can see that the constant $c$ in the definition of $\|b\|'_{\operatorname{BMO}(G)}$ can be chosen to be a median value of $b$.
 
We now recall that given a weight $\nu$, the weighted BMO space ${\rm BMO}_\nu(\mathcal G)$ is
defined as
$ {\rm BMO}_{\nu}(\mathcal G):=\{ b\in L^1_{loc}(\mathcal G):\  \|b\|_{{\rm BMO}_{\nu}(\mathcal G)}<\infty \},$
where
$$\|b\|_{{\rm BMO}_{\nu}(\mathcal G)}:=\sup_B {1\over \nu(B)}\int_B|b(g)-b_B|dg.$$ We also recall the definition of a
Muckenhoupt $A_p$ weight.
  Let $w(x)$ be a nonnegative locally integrable function
  on~$\mathcal G$. For $1 < p < \infty$, we
  say $w$ is an $A_p$ \emph{weight}, written $w\in
  A_p$, if
  \[
    [w]_{A_p}
    := \sup_B \left({1\over |B|}\int_B w(g)\,dg\right)
    \left({1\over |B|}\int_B
      \left(\dfrac{1}{w(g)}\right)^{1/(p-1)}\,dg\right)^{p-1}
    < \infty.
  \]
  Here the supremum is taken over all balls~$B\subset \mathcal G$.
  The quantity $[w]_{A_p}$ is called the \emph{$A_p$~constant
  of~$w$}.

Given $w\in A_p$ with $1<p<\infty$, the space $L^p_w(\mathcal G)$ is defined as the set of $w$-measurable functions with
$\|f\|_{L^p_w(\mathcal G)}:=\Big( \int_{\mathcal G} |f(g)|^pw(g)dg\Big)^{1\over p} <\infty$.

\medskip
We also recall the definition for Heisenberg group. Recall that $\mathbb{H}^{n}$ is
the Lie group with underlying manifold $\mathbb{C}^{n} \times \mathbb{R}= \{[z,t]: z \in \mathbb{C}^{n}, t \in \mathbb{R}\}$
and multiplication law
\begin{align*}
[z,t]\circ [z', t']&=[z_{1},\cdots,z_{n}, t] \circ [z_{1}',\cdots,z_{n}', t']\\
&: =\Big[z_{1}+z_{1}',\cdots, z_{n}+z_{n}',
t+t'+2 \mbox{Im} \big(\sum_{j=1}^{n}z_{j} \bar{z}'_{j}\big) \Big].
\end{align*}
The identity of $\mathbb{H}^{n}$ is the origin and the inverse is given
by $[z, t]^{-1} = [-z,-t]$.
Hereafter, we agree to identify $\mathbb{C}^{n}$  with $\mathbb{R}^{2n}$ and to use the following notation to denote the
points  of $\mathbb{C}^{n} \times \mathbb{R} \equiv \mathbb{R}^{2n+1}$:
$g=[z,t] \equiv [x,y,t] = [x_{1},\cdots, x_{n}, y_{1},\cdots, y_{n},t]
$
with $z = [z_{1},\cdots, z_{n}]$, $z_{j} =x_{j}+iy_{j}$ and $x_{j}, y_{j}, t \in \mathbb{R}$ for $j=1,\ldots,n$. Then, the
composition law $\circ$ can be explicitly written as
\begin{equation*}
g\circ g'=[x,y,t] \circ [x',y',t'] = [x+x', y+y', t+t' + 2\langle y, x'\rangle -2\langle x, y'\rangle],
\end{equation*}
where $\langle \cdot, \cdot\rangle$ denotes the usual inner product in $\mathbb{R}^{n}$.

The Lie algebra of the left invariant vector fields of $\H^n$ is generated by (here and in the following, we shall identify vector fields as the associated first order differential operators)
\begin{align}\label{XYT}
X_j={\partial \over \partial x_j} + 2y_j {\partial\over\partial t}, \quad Y_j={\partial \over \partial y_j} - 2x_j {\partial\over\partial t}, \quad T={\partial\over \partial t},
\end{align}
for $j=1,\ldots,n$. We now denote $X_{n+j}=Y_j$ for $j=1,\ldots,n$. Then the sub-Laplacian on $\mathbb H^n$ is $\Delta_{\mathbb H^n}=\sum_{j=1}^{2n}X_j^2$.

{The Heisenberg distance derived form the the  Kor\'anyi norm (which is also the standard homogeneous norm on $\mathbb H^n$)
\begin{align}\label{dk}
d_K(g)=(\|z\|^4+ t^2)^{1\over 4}, g=(z,t)\in\H^n,
\end{align}
is given by
\begin{align*}
d_K(g, g') = d_K(g'^{-1} \circ g) = d_K(g^{-1} \circ g'), \quad \forall g, g' \in \mathbb H^n.
\end{align*}
It is also a quasi-distance and  there exists a constant $C_K\geq1$ such that 
 \begin{align} \label{qdr1}
d_K(g_1, g_2) \leq C_{d_K}  \left( d_K(g_1, g') + d_K(g', g_2)   \right), \qquad \forall g_1, g_2, g' \in  \mathbb H^n.
\end{align}
} 

\vspace{0.2cm}

{

{\bf Notation:} 
In what follows, $C$, $C'$, etc. will denote various constants which depend only on the triple
$(\mathcal{G}, \rho, \{ \X_j \}_{1 \leq j \leq n})$. By $A\lesssim B$,
we shall mean $A\le C B$ with such a
 $C$, and $A\sim B$ stands for  $A \le C B$  and  $B\le C' A$.
}

\section{Revisit of the lower bound for kernel of Riesz transform $R_j:= \X_j (-\Delta)^{-\frac{1}{2}}$ and proof of Theorem \ref{thm0}}\label{S3}
\setcounter{equation}{0}

In this section, we study a suitable version of the lower bound for the kernel of the Riesz transform $R_j:= \X_j (-\Delta)^{-\frac{1}{2}}$, $j=1,\ldots,n$, on stratified Lie group $\mathcal G$. Here we make good use of the dilation structure on $\mathcal G$. It is not clear whether one can obtain similar lower bounds for the Riesz kernel on a general nilpotent Lie groups which is not stratified.

{To begin with, we first recall that
by \eqref{equi metric d} and the classical estimates for heat kernel and its derivations on stratified groups (see for example \cite{VSC92}), it is well-known that
\begin{align} \label{MEHT}
&|K_j(g, g')| + \rho(g, g') \sum_{i = 1}^n \left( |\X_{i, g} K_j(g, g')| + |\X_{i, g'} K_j(g, g')| \right) \nonumber \\
&\lesssim \rho(g, g')^{-Q}, \quad \forall \ 1 \leq j \leq n, \  g \neq g',
\end{align}
where $\X_{i, g}$ denotes the derivation with respect to $g$.

In \cite{DLLW}, the first four authors showed the following properties for the Riesz kernel $K_j$. 
\begin{lemma}[\cite{DLLW}]\label{lem non zero}
For all $1 \leq j \leq n$, we have $K_j \not\equiv 0$ in $\mathcal{G} \setminus \{ 0 \}$.
\end{lemma}

\color{black}
%

Now we prove Theorem \ref{thm0}.

\begin{proof}[Proof of Theorem \ref{thm0}]
For any fixed $ j\in\{1,\ldots,n \}$, by Lemma \ref{lem non zero} and the scaling property of $K_j$ (c.f. \eqref{kjs}), there exists
$\tilde g_j$ satisfying 
$$ \rho(\tilde g_j)=1\quad {\rm and}\quad K_j(\tilde g_j^{-1})\not=0. $$
Since $K_j$ is a $C^\infty$ function on $\mathcal G\backslash\{0\}$, there exists $0<\epsilon<1$ such that
\begin{align}\label{kjmore}
  K_j(g')\not=0 \quad{\rm and}\quad |K_j(g')|>{1\over2} |K_j(\tilde g_j^{-1})|  
\end{align}
for all $g'\in B(\tilde g_j^{-1},4C_\rho \epsilon)$, where $C_\rho>1$ is the constant in \eqref{qdr}. To be more specific, we have that for all $g'\in B(\tilde g_j^{-1},4C_\rho \epsilon)$, the values $K_j(g')$ and  $K_j(\tilde g_j^{-1})$ have the same sign.

Take a large $r_o$ with $r_o>{1\over \epsilon}$, we now set
\begin{align*}
G_e&:= \bigcup_{r\geq{r_o\over\alpha}}\delta_r\Big(  B( \tilde g_j,\epsilon )  \Big),
\end{align*}
where by \eqref{qdr}, for $\epsilon$ small enough, we have
$$\alpha:= \min_{ g\in \overline{ B( \tilde g_j,\epsilon ) } } \rho(g) \geq 1-C_\rho\epsilon \geq 1-{1\over100}$$
and
\begin{align*}
\beta:=\max_{ g\in \overline{ B( \tilde g_j,\epsilon ) } } \rho(g)\leq 1+C_\rho\epsilon \leq 1+{1\over100}.
\end{align*}


Now for any $g\in \mathcal G$, we let 
\begin{align*}
G:=\tau_g(G_e).
\end{align*}
It is clear that 
$$\inf_{g'\in G}\rho(g,g')=r_o.$$
We now show that the set $G$ defined as above satisfies all the required conditions. 

First, we point out that for any $g_2\in G$ and for any $g_1\in B(g,1)$,
following similar calculation and estimates in \cite[Theorem 1.1]{DLLW}, we obtain that 
$$|K_j(g_1,g_2)| 
\geq {1\over 2} r^{-Q}|K_j(\tilde g_j^{-1})|,\quad
\big|K_j(g_2,g_1)\big| \geq {1\over 2} r^{-Q}\big|K_j(\tilde g_j^{-1})\big|.$$
Moreover, all the $K_j(g_1,g_2)$ as well as all the $K_j(g_2,g_1)$, $g_1\in B(g,1), g_2\in G$, have the same sign.
We denote 
$$C:= \min_{1\leq j\leq n} {1\over 2} \big|K_j(\tilde g_j^{-1})\big|,$$
then \eqref{lower bound} holds with the constant $C$ defined above.

We now only show that for any $R>2r_o$,
\eqref{regular} holds, since the other properties are obvious. In fact, we note that
$$ B(g,R)\supset B(g,R)\cap G = \tau_g\big(B(0,R)\cap G_e\big). $$
Moreover, we have
\begin{align*}
B(0,R)\cap G_e &\supset \bigcup_{ {r_o\over\alpha}\leq r\leq { {1-{1\over100}}\over {1+\beta}}R } \delta_r\Big( {B(\tilde g_j,\epsilon)} \Big)\supset \delta_{{ {1-{1\over100}}\over {1+\beta}}R }\Big( {B(\tilde g_j,\epsilon)} \Big).
\end{align*}
As a consequence, we have
\begin{align*}
|B(0,1)|R^Q &=|B(g,R)|  \geq |B(g,R)\cap G|\\
&\geq \Big|\delta_{{ {1-{1\over100}}\over {1+\beta}}R }\Big( {B(\tilde g_j,\epsilon)} \Big)\Big|
\\
&\geq C_\epsilon R^Q,
\end{align*}
which shows that \eqref{regular} holds.
This completes the proof of Theorem \ref{thm0}.
\end{proof}

By using $r_or$ to replace $r_o$ in the above proof, 
we can also get the similar result for any ball $B(g, r)$.
 \begin{corollary}\label{cor1}
Suppose that $\mathcal G$ is a stratified Lie group with homogeneous dimension $Q$ and that $j\in\{1,2,\ldots,n\}$.  
There exist a large positive constant $r_o$ and a positive constant $C$ depending on $K_j$ such that for every 
$g\in \mathcal G$ there exists a set $G\subset \mathcal G$ such that  $\inf\limits_{g'\in G} \rho(g,g')=r_or$ and that  for every $g_1\in B(g,r)$ and $g_2\in G$, we have
$$ |K_j(g_1,g_2)|\geq C  \rho(g_1,g_2)^{-Q}, \quad
  |K_j(g_2,g_1)|\geq C  \rho(g_1,g_2)^{-Q},$$
all $K_j(g_1,g_2)$ as well as all $K_j(g_2,g_1)$ have the same sign.  

Moreover, the set $G$ is regular, in the sense that $|G|=\infty$ and that for any $R>2r_or$,
$$ |B(g,R) \cap G|\approx R^Q,  $$
where the implicit constants are independent of $g$ and $R$. 
\end{corollary}

\section{Two weight estimates for commutator $[b,R_j]$ and the proof of Theorem \ref{thm two weight}}
\setcounter{equation}{0}

To begin with, we point out that by considering the stratified Lie group $\mathcal G$ as a space of homogeneous type in the sense of Coifman and Weiss with the metric $\rho$, we have a system  $\mathcal D(\mathcal G)$ of dyadic cubes on $\mathcal G$. We refer to the original construction from \cite{Chr} and the refinement from \cite{HK}. See also \cite[Section 2]{KLPW} for a summary. 


Next we also recall that there exist adjacent systems of dyadic cubes on $\mathcal G$, denoted by $\mathcal D^1(\mathcal G)$, $\ldots$, $\mathcal D^{\mathcal T} (\mathcal G)$, such that 
for each ball $B\subset \mathcal G$, there exist $t\in\{1,2,\ldots, \mathcal T\}$ and $S\in \mathcal D^t(\mathcal G)$ satisfying
$$ B\subset S\subset C B, $$ 
where $C$ is an absolute constant independent of $t$ and $S$, and $CB$ denotes the ball with the same center as $B$ and radius $C$ times that of $B$. See  \cite{HK} and \cite[Section 2.4]{KLPW} for more details. Associated to each systems of dyadic cubes, one has the dyadic BMO space as follows. A dyadic weighted BMO space associated with the system $\mathcal D^t(\mathcal G)$   is
defined as
${\rm BMO}_{\nu,\mathcal D^t(\mathcal G)}(\mathcal G):=\{ b\in L^1_{loc}(\mathcal G):\  \|b\|_{{\rm BMO}_{\nu,\mathcal D^t(\mathcal G)}(\mathcal G)}<\infty \},$
where
$\|b\|_{{\rm BMO}_{\nu,\mathcal D^t(\mathcal G)}(\mathcal G)}:=\sup_{S\in \mathcal D^t(\mathcal G) } {1\over \nu(S)}\int_S|b(g)-b_S|dg.$ 

Then according to the dyadic structure theorem studied in \cite{HK,KLPW}, one has
$$  {\rm BMO}_{\nu}(\mathcal G)=\bigcap_{t=1}^\mathcal T {\rm BMO}_{\nu,\mathcal D^t(\mathcal G)}(\mathcal G).$$
Thus, to verify a function $b$ is in ${\rm BMO}_{\nu}(\mathcal G)$, it suffices to verify it belongs to 
each weighted dyadic BMO space ${\rm BMO}_{\nu,\mathcal D^t(\mathcal G)}(\mathcal G)$.

Given a dyadic cube $S\in \mathcal D^t(\mathcal G)$ with $t=1,\ldots, \mathcal T$, and a measurable function $f$ on $\mathcal G$, we define 
the local mean oscillation of $f$ on $S$ by
$$ w_\lambda(f;S)=\inf_{c\in\mathbb R} \big( (f-c)\chi_S \big)^*(\lambda |S|), \quad 0<\lambda<1, $$
where $f^*$ denotes the non-increasing rearrangement of $f$.

With these notation and dyadic structure theorem above, following the same proof in \cite[Lemma 2.1]{Ler}, we also obtain that 
for any weight $\nu\in A_2(\mathcal G)$, we have
\begin{align}\label{BMO norm}
\|b\|_{ {\rm BMO}_\nu(\mathcal G)  }\leq C \sum_{t=1}^\mathcal T \|b\|_{{\rm BMO}_{\nu,\mathcal D^t(\mathcal G)}(\mathcal G)}   \leq C\sum_{t=1}^\mathcal T \sup_{S\in \mathcal D^t(\mathcal G)} w_\lambda(b;S) {|S|\over \nu(S)}, \quad 0<\lambda\leq {1\over 2^{Q+2}},
\end{align}
where $C$ depends on $\nu$.

Next, we point out that from the construction of dyadic system $\mathcal D^t(\mathcal G)$ as in \cite{Chr, HK}, for every dyadic cube $S\in \mathcal D^t(\mathcal G)$ in level $l$, there exist two balls $B_1$ and $B_2$ with radius $r_1$ and $r_2$, respectively, such that
\begin{align}\label{ball in and out}
 B_1\subset S\subset B_2 
  \end{align}
and that $C_1r_1 \leq 2^{-l} \leq \overline C_1 r_1$ and $C_2r_2 \leq 2^{-l} \leq \overline C_2 r_2$, where the  constants $C_1,\overline C_1, C_2$ and $\overline C_2$ are independent of $r_1,r_2$ and $l$.

To prove Theorem \ref{thm two weight}, we first need to establish the following result.
\begin{proposition}\label{prop w}
Suppose that $\mathcal G$ is a stratified Lie group with homogeneous dimension $Q$ as defined in Section 2, $b\in L^1_{loc}(\mathcal G)$ and that $K_j$ is the kernel of the $j$th Riesz transform on $\mathcal G$, $j=1,2,\ldots,n$. 
Let $r_o$ be the constant in Corollary \ref{cor1}.
Then for any $k_0>r_o$ and for any dyadic cube  $S\in \mathcal D(\mathcal G)$, there exist measurable sets $E\subset S$ and
$F\subset k_0B_1$ with $B_1$ the ball in \eqref{ball in and out}, such that
\begin{enumerate}
\item $|E\times F| \sim |S|^2$,

\item $w_{1\over 2^{Q+2}}(b;S)\leq |b(g)-b(g')|$,\quad $\forall (g,g')\in E\times F$,

\item $K_j(g,g')$ and $ b(g)-b(g') $ do not change sign for any $(g,g')\in E\times F$,

\item $ |K_j(g,g')|\geq C  \rho(g,g')^{-Q}$ for any $(g,g')\in E\times F$.

\end{enumerate}

\end{proposition}
\begin{proof}
From the fact \eqref{ball in and out}, for any dyadic cube  $S\in \mathcal D^t(\mathcal G)$ with $t=1,2,\ldots,\mathcal T$, we now consider
the ball $B_2$ containing $S$ with radius comparable to the side-length of $S$. For simplicity, we denote it by $B=B(x_0,r)$.

From Corollary \ref{cor1}, we have that there exists a large positive constant $r_o$ such that for this
$B(x_0,r)$, there exists a set $G\subset\mathcal G$ such that 
$ \inf\limits_{g'\in G} \rho(g,g')=r_or$ and $K_j(g,g')$  do not change sign for any $(g,g')\in B(x_0,r)\times G$. Moreover, (4) holds.

Now to show the other three properties, we need to consider an appropriate subset of $G$. We define it as follows.
For any $k_0>r_o$, we let $ F_{k_0} := k_0B\cap G $. Then by using  Corollary \ref{cor1} again we have that
\begin{align}\label{Fk0}
|F_{k_0}| \sim k_0^Q r^Q \sim k_0^Q |B|.
\end{align} 

From the definition of $w_{1\over 2^{Q+2}}(b;S)$ we see that there exists a subset $\mathcal E\subset S$ with 
$ |\mathcal E|= {1\over 2^{Q+2}}|S| $ such that for any $g\in\mathcal G$,
\begin{align}
w_{1\over 2^{Q+2}}(b;S)\leq  |b(g) - m_b(F_{k_0})|.
\end{align} 

Next, we show that there exist $E\subset\mathcal E$ and $F\subset F_{k_0}$ such that 
$ |E|={1\over 2^{Q+3}}|S| $, $|F| = {1\over 2} |F_{k_0}|$ and that
\begin{align}\label{b(g)}
|b(g) - m_b(F_{k_0})|\leq | b(g) -b(g') |, \quad \forall (g,g')\in E\times F
\end{align} 
and moreover, $b(g) -b(g')$ does not change sign in $E\times F$.

To see this, we let 
\begin{align*}
E_1&=\{ g\in\mathcal E:\ b(g)\geq m_b(F_{k_0}) \},\quad E_2=\{ g\in\mathcal E:\ b(g)\leq m_b(F_{k_0}) \};\\
F_1&=\{ g'\in F_{k_0}:\ b(g')\geq m_b(F_{k_0}) \},\quad F_2=\{ g'\in F_{k_0}:\ b(g')\leq m_b(F_{k_0}) \}.
\end{align*} 
Then we have 
$$ |F_1|\geq {1\over 2}|F_{k_0}|, \quad   |F_2|\geq {1\over 2}|F_{k_0}|,$$
and there exists $i\in\{1,2\}$ such that $|E_i|\geq {1\over2}|\mathcal E|$. Without loss of generality we assume
$ |E_1|\geq {1\over2}|\mathcal E  |$.
Hence, there exist $E\subset E_1$ and $F\subset F_1$ such that 
$$   |E|={1\over 2}|\mathcal E| \quad {\rm and}\quad  |F|={1\over 2}| F_{k_0}|. $$
Thus, for any $(g,g')\in E\times F$, we have
$$  |b(g) - m_b(F_{k_0})|  = b(g) - m_b(F_{k_0}) \leq b(g)-b(g'), $$
which implies that \eqref{b(g)} holds and that $b(g)-b(g')$ does not change sign in $E\times F$.

As a consequence, we get that (2) and (3) hold. Next, from \eqref{Fk0}, we get that
$$   |E\times F| =|E|\times |F| ={1\over 4} |\mathcal E|\cdot |F_{k_0}| ={1\over 2^{Q+5}} |S|\cdot|F_{k_0}| \sim k_0^Q|S|^2,   $$
which shows that (1) holds.

The proof of Proposition \ref{prop w} is complete.
\end{proof}

\begin{proof}[Proof of Theorem \ref{thm two weight}]
Note that it suffices to prove (ii), since (i) follows from \cite{HLW} or \cite{Ler} using the size and smoothness of Riesz transform kernel only.

To prove (ii), based on \eqref{BMO norm}, it suffices to show that there exists a positive constant $C$ such that for all
dyadic cubes $S\in\mathcal D(\mathcal G)$,
\begin{align}\label{w norm}
w_{1\over 2^{Q+2}}(b;S) \leq C{\nu(S)\over |S|}{\|[b,R_j]\|_{L^p_\mu(\mathcal G)\to L^p_\lambda(\mathcal G)}}.
\end{align}

To see this, we first note that  property (2) in Proposition \ref{prop w} implies that
\begin{align*}
w_{1\over 2^{Q+2}}(b;S) |E\times F|\leq \iint_{E\times F} |b(g)-b(g')|\, dgdg'.
\end{align*}
From this, using property (4) in Proposition \ref{prop w} and the fact that $\rho(g,g')\leq \overline C(k_0+1){\rm diam}(S)$ for all $(g,g')\in E\times F$ with the constant $\overline C$ depending only on $C_2$ and $\overline C_2$ in the inclusion \eqref{ball in and out}, we obtain that
\begin{align*}
w_{1\over 2^{Q+2}}(b;S) |E\times F|\leq C |S|\iint_{E\times F} |b(g)-b(g')|{1\over  \rho(g,g')^{Q}} \, dgdg'.
\end{align*}
From property (3) in Proposition \ref{prop w}, we get that $K_j(g,g')$ and $ b(g)-b(g') $ do not change sign for any $(g,g')\in E\times F$. Hence, taking into account the property (1) in Proposition \ref{prop w}, we have
\begin{align*}
w_{1\over 2^{Q+2}}(b;S) &\leq C{1\over |S|}\iint_{E\times F} |b(g)-b(g')|\, |K_j(g,g')| \, dgdg'\\
&= C{1\over |S|}\bigg|\iint_{E\times F} (b(g)-b(g')) \,K_j(g,g')\, dgdg' \bigg|\\
&\leq C{1\over |S|} \int_E \big| [b,R_j](\chi_F)(g) \big|\,dg.
\end{align*}
From H\"older's inequality, we further have
\begin{align*}
w_{1\over 2^{Q+2}}(b;S) 
&\leq C{1\over |S|} \bigg(\int_E \big| [b,R_j](\chi_F)(g) \big|^p \lambda(g)\,dg\bigg)^{1\over p} \bigg( \int_S \lambda^{-{1\over p-1}}(g)dg\bigg)^{1-{1\over p}}\\
&\leq C{1\over |S|} \mu(F)^{1\over p} \bigg( \int_S \lambda^{-{1\over p-1}}(g)dg\bigg)^{1-{1\over p}}{\|[b,R_j]\|_{L^p_\mu(\mathcal G)\to L^p_\lambda(\mathcal G)}}\\
&\leq C{1\over |S|} \mu(S)^{1\over p} \bigg( \int_S \lambda^{-{1\over p-1}}(g)dg\bigg)^{1-{1\over p}}{\|[b,R_j]\|_{L^p_\mu(\mathcal G)\to L^p_\lambda(\mathcal G)}}\\
&\leq C{\nu(S)\over |S|}{\|[b,R_j]\|_{L^p_\mu(\mathcal G)\to L^p_\lambda(\mathcal G)}},
\end{align*}
where the last inequality follows from the definition of the weight $\nu$. This proves \eqref{w norm}.  The proof of Theorem \ref{thm two weight} is complete.
\end{proof}

\section{Endpoint characterisation of BMO$(\mathcal G)$ via the $L\log^+L\to L^{1,\infty}$ boundedness of the commutator $[b,R_j]$ and the proof of Theorem \ref{thm1}}
\setcounter{equation}{0}

\begin{proof}[Proof of Theorem \ref{thm1}]

\bigskip
For the sufficient part, we point out that it follows directly from \cite{P} with only minor changes, since 
the whole proof can be adapted from Euclidean space to stratified   Lie groups and
the key conditions for the operator $T$ in \cite{P} are the upper bound of size and smoothness properties of the kernel; all of which we have in the setting at hand.

For the necessity part, for any ball $B=B(g_0, r)\subset\mathcal G$, define
$$M(b,B)=\inf_{c}{1\over |B|}\int_{B}|b(g)-c|dg$$ and
we now prove
\begin{align}\label{mb}
\sup_{B}M(b,B)\leq C(Q,\theta_b).
\end{align}

We claim that it suffices to prove \eqref{mb} for the ball $B(0,1)$. To see this,
for a measurable function $f$ on $\mathcal G$, $g_0\in\mathcal G$ and $r>0$, we define the translation and dilation
of $f$ by
$$  \tau_{g_0}(f)(g) := f(g_0\circ g),\quad  \delta_r(f)(g) := f\big(\delta_r(g)\big).  $$
Remark that the Riesz transforms are translation-invariant and satisfy the scaling property, namely, for any $g\in\mathcal G$, $r>0$ and for every $j=1,2,\ldots,n$,
$$  \tau_{g}\big(R_j(f)\big) = R_j\big( \tau_g(f) \big),\quad  \delta_r\big(R_j(f)\big) =  R_j\big( \delta_r(f) \big).$$
So we have
$$ \big[ \tau_{g_0}\circ \delta_r(b), R_j \big](f) = \tau_{g_0}\circ \delta_r\Big( [b,R_j]\big( \delta_{1\over r} \circ \tau_{g_0^{-1}}  (f)\big) \Big).  $$
Thus, it suffices to prove \eqref{mb} for the ball $B(0,1)$.

Let 
$$M:={1\over |B(0,1)|}\int_{B(0,1)}\left|b(g)-m_b(B(0,1))\right|dg,$$
where $m_b(B(0,1))$ is the median of $b$ over $B(0,1)$ as in Definition \ref{def median}. Since
$$[b-m_b(B(0,1)), R_j]=[b, R_j],$$ without loss of generality, we may assume that $m_b(B(0,1))=0$. This means that we can find disjoint subsets $E_1, E_2\subset B(0,1)$ such that 
\begin{align*}
E_1\supset\{g\in B(0,1): b(g)<0\},\quad E_2\supset\{g\in B(0,1): b(g)>0\},
\end{align*}
and
$|E_1|=|E_2|={1\over 2} \left|B(0,1)\right|.$

Define 
$\varphi(g)=\chi_{E_2}(g)-\chi_{E_1}(g).$
Then $\varphi$ satisfies $\supp \varphi\subset B(0,1)$, 
$$\|\varphi\|_{L^{\infty}(G)}=1, \quad \varphi(g)b(g)\geq0, \quad \int_{B(0,1)} \varphi(g)dg=0,$$
 and
  $${1\over |B(0,1)|}\int_{B(0,1)}\varphi(g)b(g)dg=M.$$

In the following, for $i=1,\cdots,10$, $A_i$ denotes a positive constant depending only on $K_j, Q$, $C_\rho$ in \eqref{qdr} and $A_l, 1\leq l<i$.
For the ball $B(0,1)$, take the set $G$ as in Theorem \ref{thm0}. For $g\in G$,
\begin{align*}
\big|\left[b, R_j\right]\varphi(g)\big|=\big|b(g)R_j(\varphi)(g)-R_j(b\varphi)(g)\big|
\geq \big|R_j(b\varphi)(g)\big|-|b(g)| \big|R_j(\varphi)(g)\big|.
\end{align*}
We estimate these two terms separately. By Theorem \ref{thm0},
\begin{align*}
 \big|R_j(b\varphi)(g)\big|
 =\int_{B(0,1)}\left|K_j(g,g')\right| \left|b(g')\right|dg'
 \geq 
 A_1 M \rho(g)^{-Q}.
\end{align*}

For the second term, since $ \int_{B(0,1)} \varphi(g)dg=0$, by \eqref{MEHT}, we have
\begin{align*}
\big|R_j(\varphi)(g)\big| 
&\leq \int_{B(0,1)}\left|K_j(g,g')-K_j(g,0)\right|\left|\varphi(g')\right|dg'
\leq A_2 \rho(g)^{-Q-1}.
\end{align*}
Thus, we have
\begin{align*}
\big|\left[b, R_j\right]\varphi(g)\big|
\geq A_1 M \rho(g)^{-Q}-A_2|b(g)| \rho(g)^{-Q-1}.
\end{align*}

Let 
$$F:=\left\{g\in G: |b(g)|>{A_1\over 2A_2} M \rho(g)\ \ {\rm{and}} \ \ \rho(g)<M^{1\over Q}\right\},$$
then by Theorem \ref{thm0}, we have
\begin{align}\label{m1}
\left|\left\{g\in\mathcal G:\big|\left[b, R_j\right]\varphi(g)\big|>{A_1\over 2}\right\}\right|\nonumber
&\geq\left|\left(G\setminus F\right)\cap B\big(0,M^{1\over Q}\big)\right|\nonumber\\
&\geq A_3 M-|F|,
\end{align}
where the last inequality is due to assuming $M>(2r_o)^{Q}$.

By assumption, we also have
\begin{align}\label{m2}
\nonumber\left|\left\{g\in\mathcal G:\big|\left[b, R_j\right]\varphi(g)\big|>{A_1\over 2}\right\}\right|
\nonumber&\leq\theta_b\int_{B(0,1)}{2|\varphi(g)|\over A_1}\left(1+\log^+\left({2|\varphi(g)|\over A_1}\right)\right)dg\\
&={2\theta_b\over A_1}\left(1+\log^+\left({2\over A_1}\right)\right).
\end{align}
Then \eqref{m1} and \eqref{m2} imply that 
$$|F|\geq A_3 M-{2\theta_b\over A_1}\left(1+\log^+\left({2\over A_1}\right)\right)\geq {A_3\over 2}M,$$
by assuming that $M>{4\theta_b\over A_1A_3}(1+\log^+({2\over A_1}))$.

Let $\psi(g):=\operatorname{sgn}(b(g))\chi_{F}(g)$, then for $g\in B(0,1)$, 
$$\big|\left[b, R_j\right]\psi(g)\big|=\big|b(g)R_j(\psi)(g)-R_j(b\psi)(g)\big|
\geq \big|R_j(b\psi)(g)\big|-|b(g)| \big|R_j(\psi)(g)\big|.$$


From Theorem \ref{thm0}, we have
\begin{align*}
 \big|R_j(b\psi)(g)\big|&=\left|\int_{\mathcal G}K_j(g,g')b(g')\psi(g')dg'\right|
 =\int_{F}\left|K_j(g,g')\right| \left|b(g')\right|dg'\\
 &\geq A_4\int_{F}\rho(g')^{-Q} \left|b(g')\right|dg'
 \geq {A_4A_1\over 2A_2}\int_{F}M\rho(g')^{-Q+1}dg'\\
 &\geq  {A_4A_1\over 2A_2}M^{1\over Q}|F|\\
 &\geq A_5 M^{1+{1\over Q}}.
 \end{align*}

 For the second term, by \eqref{MEHT}, for $g\in B(0,1)$, we have
 \begin{align*}
\big|R_j(\psi)(g)\big|& 
\leq \int_{F}\left|K_j(g,g')\right|\left|\psi(g')\right|dg'
\leq A_6\int_{F}\rho(g')^{-Q}dg'\leq A_7\log M.
\end{align*}
Therefore, for $g\in B(0,1)$, we have
\begin{align*}
\big|\left[b, R_j\right]\psi(g)\big|
\geq A_5 M^{1+{1\over Q}}-A_7\left|b\left(g\right)\right|\log M.
\end{align*}

 By our assumption on $R_j$, we have
\begin{align*}
\left|\left\{g\in\mathcal G:\big|\left[b, R_j\right]\psi(g)\big|>{A_5\over 2}M^{1+{1\over Q}}\right\}\right|
&\leq\theta_b\int_{\mathcal G}{2|\psi(g)|\over A_5 M^{1+{1\over Q}}}\Big(1+\log^+\Big({2|\psi(g)|\over A_5 M^{1+{1\over Q}}}\Big)\Big)dg\\
&\leq A_8\theta_bM^{-{1\over Q}},
\end{align*}
where the last inequality follows by taking $M$ large enough ($M>({2\over A_5})^{Q\over Q+1}$).

On the other hand, 
\begin{align*}
A_8\theta_bM^{-{1\over Q}}
&\geq\left|\left\{g\in B(0,1):\big|\left[b, R_j\right]\psi(g)\big|>{A_5\over 2}M^{1+{1\over Q}}\right\}\right|\\
&\geq\left|\left\{g\in B(0,1):A_5 M^{1+{1\over Q}}-A_7\left|b(g)\right|\log M
>{A_5\over 2}M^{1+{1\over Q}}\right\}\right|\\
&=|B(0,1)|-\left|\left\{g\in B(0,1): |b(g)|\geq A_9M^{1+{1\over Q}}(\log M)^{-1}\right\}\right|\\
&\geq 
|B(0,1)|-|B(0,1)| A_9^{-1}M^{-{1\over Q}}\log M\\
&\geq A_{10},
\end{align*}
where the last inequality comes from the fact that $M^{-{1\over Q}}\log M<Qe^{-1}$ whenever $M>e^{Q}$.
Therefore, we have 
$$M\leq \Big({A_8\over A_{10}}\Big)^Q\theta_b^Q.$$
Summing up the above estimates, we can obtain that
\begin{align*}
M&\leq\max\Big\{ (2r_o)^{Q}, {4\theta_b\over A_1A_3}\Big(1+\log^+\Big({2\over A_1}\Big)\Big),
 \Big({2\over A_5}\Big)^{Q\over Q+1}, \ \ e^{Q}, \ \  \Big({\theta_bA_8\over A_{10}}\Big)^Q\Big\}=: C(Q,\theta_b).
 \end{align*}
The proof of Theorem \ref{thm1} is complete. 
\end{proof}

\section{Endpoint characterisation of commutator $[b,R_j]$ via $H^1(\mathcal G)$ and BMO$(\mathcal G)$ and proof of Theorem \ref{thm2}}
\setcounter{equation}{0}

Recall that the Hardy space $H^1(\mathcal G)$ can be characterized by the atomic decomposition \cite{FoSt}.
\begin{definition}\label{def Hardy atom}
 The space $H^1(\mathcal G)$ is the set of functions of the form $f=\sum_{j=1}^\infty \lambda_j a_j$ with $\{\lambda_j\}\in\ell^1$ and $a_j$ a $(1,q)$ atom, $1<q\leq\infty$, meaning that it is supported on a ball $B\subset \mathcal G$, has mean value zero $\int_B a(g)dg=0$ and has a size condition $\left\Vert a\right\Vert_{L^q(\mathcal G)}\leq |B|^{{1\over q}-1}$.  The norm of $H^1(\mathcal G)$ is defined by:
$$
\left\Vert f\right\Vert_{H^1(\mathcal G)}:=\inf\bigg\{\sum_{j=1}^\infty \left\vert \lambda_j\right\vert: \{\lambda_j\}\in \ell^1, f=\sum_{j=1}^\infty \lambda_j a_j, a_j \textnormal{ a $(1,q)$-atom}\bigg\}
$$
with the infimum taken over all possible representations of $f$ via atomic decompositions.
\end{definition}
We point out that for any $q\in (1,\infty]$, the definitions of $H^1(\mathcal G)$ via these  $(1,q)$-atoms are equivalent.

\begin{proposition}\label{prop1}
Suppose that $\mathcal G$ is a stratified Lie group, $b\in {\rm BMO}(\mathcal G)$  and $j\in\{1,2,\ldots,n\}$. Then 
the following statements are equivalent:
\begin{itemize}
\item[(i)] $[b,R_j]$ is bounded from $H^1(\mathcal G)$ to $L^1(\mathcal G)$;
\item[(ii)]  $b$ satisfies the following condition:
for any $(1,p)$-atom $a$ with $1<p<\infty$ supported in a ball $B$ and $\tilde g\in B$,
\begin{align}\label{h1b}
\left(\int_{(r_oB)^c}\left| K_j(g,\tilde g)\right|dg\right)\left|\int_{B}b(g')a(g')dg'\right|\leq C,
\end{align}
where $r_o$ is the one in Theorem \ref{thm0}.
 \end{itemize}
\end{proposition}


\begin{proof}
Note that $b\in {\rm BMO}(\mathcal G)$. We have that   $[b,R_j]$ is bounded on $L^p(\mathcal G)$ for $1<p<\infty$.
Assume that $a$ is a $(1,p)$-atom which is supported in some ball $B$, then by \cite[Theorem 1.2]{DLLW}, we can see that $[b,R_j](a)$ makes sense and belongs to $ L^p(\mathcal G)$.

For any $g\in\mathcal G$, we can write
\begin{align}\label{decom}
[b,R_j](a)(g) 
&=\chi_{r_oB}(g)[b,R_j](a)(g)+\chi_{(r_oB)^c}(g)(b(g)-b_B)R_j(a)(g)\\
\nonumber&\quad -\chi_{(r_oB)^c}(g)\int_{\mathcal G}\left(K_j(g,g')-K_j(g,\tilde g)\right)\left(b(g')-b_B\right)a(g')dg'\\
\nonumber&\quad-\chi_{(r_oB)^c}(g)\int_{\mathcal G}K_j(g,\tilde g)\left(b(g')-b_B\right)a(g')dg'\\
\nonumber&=:I_1(g)+I_2(g)+I_3(g,\tilde g)+I_4(g,\tilde g),
\end{align}
where $\tilde g$ is any point in $B$.

For  $I_1$, by H\"older's inequality, \cite[Theorem 1.2]{DLLW} and the definition of atom, we have
\begin{align}\label{i1}
\|I_1\|_{L^1(\mathcal G)} 
&\leq |r_oB|^{1-{1\over p}}\left(\int_{r_oB}\big|[b,R_j](a)(g)dg\big|^pdg\right)^{1\over p}\\
\nonumber&\leq C|B|^{1-{1\over p}}\|a\|_{L^p(\mathcal G)}\leq C|B|^{1-{1\over p}}|B|^{{1\over p}-1}\nonumber\\
&=C,\nonumber
\end{align}
for any $1<p<\infty$.

By the method of choosing $r_o$ in Theorem \ref{thm0}, we can assume $r_o=2^{\gamma}$ for some $\gamma>1$ and $\gamma\in\mathbb N$.

For the term $I_2$, since $a$ has mean value zero, by \eqref{MEHT}, we have
\begin{align*}
\|I_2\|_{L^1(\mathcal G)}
&\leq \int_{(r_oB)^{c}}\big|b(g)-b_B\big|\left|\int_{B }\left(K_j(g,g')-K_j(g,\tilde g)\right)a(g')dg'\right|dg\\
&\leq\sum_{l=\gamma+1}^{\infty} \int_{(2^{l}B\setminus 2^{l-1}B)}\big|b(g)-b_B\big|
\bigg(\int_{B}\Big({\rho(g',\tilde g)\over \rho(g,g')^{Q+1}} \Big)^{p'} dg'\bigg)^{1\over p'} \|a\|_{L^p(\mathcal G)}dg\\
&\leq C \sum_{l=\gamma+1}^{\infty} 2^{-l}{1\over |2^l B|}\int_{2^{l}B}\big|b(g)-b_B\big|dg,
\end{align*}
where ${1\over p}+{1\over p'}=1$. By noting that $\big|b_{2^{i}B}-b_{2^{i-1}B}\big|\leq 2^Q\|b\|_{{\rm BMO}(\mathcal G)}$, we get that
%
%
%
\begin{align}\label{i2}
\|I_2\|_{L^1(\mathcal G)}
&\leq C \sum_{l=\gamma+1}^{\infty} 2^{-l} 2^Ql\|b\|_{{\rm BMO}(\mathcal G)}=C.
\end{align}

For the term $I_3$, by using \eqref{MEHT}, H\"older's inequality and the $L^p$ norm of the atom $a$,  we have
\begin{align} \label{i3}
\|I_3\|_{L^1(\mathcal G)} 
&\leq\sum_{l=\gamma+1}^{\infty} \int_{(2^{l}B\setminus 2^{l-1}B)}
\int_{B}{d(g',\tilde g)\over d(g,g')^{Q+1}}\big|b(g')-b_B\big|\left|a(g')\right|dg'dg \\
&\leq C \sum_{l=\gamma+1}^{\infty} 2^{-l}\int_{B}\big|b(g')-b_B\big|\left|a(g')\right|dg' \nonumber\\
&\leq C \sum_{l=\gamma+1}^{\infty} 2^{-l}\left({1\over |B|}\int_B\big|b(g')-b_B\big|^{p'}dg'\right)^{1\over p'}\nonumber\\
&\leq C \sum_{l=\gamma+1}^{\infty} 2^{-l}\|b\|_{{\rm BMO}(\mathcal G)}\nonumber\\
&\leq C,\nonumber
\end{align}
where the fourth inequality follows from the John--Nirenberg inequality for BMO space.

For the term $I_4$, by the mean value zero property of $a$, we have 
\begin{align}\label{i4}
\nonumber\|I_4\|_{L^1(\mathcal G)}&=\int_{(r_oB)^{c}}\left|\int_{B }K_j(g,\tilde g)
(b(g')-b_B)a(g')dg'\right|dg\\
&=\left(\int_{(r_oB)^c}\big| K_j(g,\tilde g)\big|dg\right)\left|\int_{B}b(g')a(g')dg'\right|.
\end{align}

From \eqref{decom}, \eqref{i1}, \eqref{i2} and \eqref{i3}, we can see that, for any $(1,p)$-atom $a$, 
$$\|[b,R_j](a)\|_{L^1(\mathcal G)}\leq C$$ if and only if $\|I_4\|_{L^1(\mathcal G)}\leq C$. Then Proposition  \ref{prop1} follows from \eqref{i4}.
\end{proof}

\begin{proposition}\label{prop2}
Suppose that $\mathcal G$ is a stratified Lie group, $b\in {\rm BMO}(\mathcal G)$  and $j\in\{1,2,\ldots,n\}$. Then 
the following statements are equivalent:
\begin{itemize}
\item[(i)] $[b,R_j]$ is bounded from $L_c^\infty(\mathcal G)$ to ${\rm BMO}(\mathcal G)$;
\item[(ii)]  $b$ satisfies the following condition:  For any ball $B$, any  $\tilde g\in B$ and  $f\in L_c^\infty(\mathcal G)$,
\begin{align}\label{lb}
\left({1\over |B|}\int_{B}\big|b(g)-b_B\big|dg\right)\left|\int_{(r_oB)^c}K_j (\tilde g, g')f(g')dg'\right|\leq C\|f\|_{L^{\infty}(\mathcal G)},
\end{align}
where $r_o$ is the one in Theorem \ref{thm0}.
 \end{itemize}
\end{proposition}

\begin{proof}
Let $f$ be a bounded function with compact support, then $f\in L^p(\mathcal G)$. Since 
$b\in {\rm BMO}(\mathcal G)$, by \cite[Theorem 1.2]{DLLW} we can see that $[b,R_j]f\in L^p(\mathcal G)$ for any $1<p<\infty$. Thus, $[b,R_j]f$ is a locally integrable function.  Fix any ball $B\in\mathcal G$, we can write
$$f=f\chi_{r_o B}+f\chi_{(r_o B)^c}=:f_1+f_2,$$
 where $r_o$ is the constant in Theorem \ref{thm0}. Then for $g\in B$, and for 
%
%
any $\tilde g\in B$, we have
  \begin{align*}
& [b,R_j](f)(g)-\big([b,R_j]f\big)_B\\
&=[b,R_j](f_1)(g)-\big([b,R_j]f_1\big)_B+\big(b(g)-b_B\big)\big(R_j(f_2)(g)-R_j(f_2)(\tilde g)\big)\\
&\quad-{1\over |B|}\int_B  \left(b(g')-b_B\right)\left(R_j(f_2)(g')-R_j(f_2)(\tilde g)\right)dg'\\
&\quad +{1\over |B|}\int_B  \left(R_j\big((b-b_B)f_2\big)(g')-R_j\big((b-b_B)f_2\big)(g)\right)dg'\\
&\quad +\big(b(g)-b_B\big)R_j(f_2)(\tilde g).
 \end{align*}

 Set
 \begin{align*}
 J_1(g)&=[b,R_j](f_1)(g),\\
J_2(g,\tilde g)&=\big(b(g)-b_B\big)\big(R_j(f_2)(g)-R_j(f_2)(\tilde g)\big),\\
J_3(g',g)&=R_j\big((b-b_B)f_2\big)(g')-R_j\big((b-b_B)f_2\big)(g),\\
J_4(g,\tilde g)&=\big(b(g)-b_B\big)R_j(f_2)(\tilde g).
 \end{align*}
 Then we have
  \begin{align}\label{deco0}
[b,R_j](f)(g)-\big([b,R_j]f\big)_B&=J_1(g)-\big(J_1\big)_B+J_2(g,\tilde g)-\big(J_2(\cdot,\tilde g)\big)_B\\
 \nonumber &\quad+\left(J_3(\cdot, g)\right)_B+J_4(g,\tilde g).
 \end{align}
Therefore,
   \begin{align}\label{deco1}
 &{1\over |B|}\int_B\big|[b,R_j](f)(g)-\big([b,R_j]f\big)_B\big|dg\\
 \nonumber &\leq {2\over |B|}\int_B|J_1(g)|dg+{2\over |B|}\int_B\left|J_2(g,\tilde g)\right|dg
 +{1\over |B|}\int_B\left|\left(J_3(\cdot, g)\right)_B\right|dg\\
  \nonumber&\quad+{1\over |B|}\int_B\left|J_4(g,\tilde g)\right|dg\\
  \nonumber&=:2L_1+2L_2(\tilde g)+L_3+L_4(\tilde g).
 \end{align}


For the term $L_1$, since $b\in\operatorname{BMO}(\mathcal G)$, by H\"older's inequality and \cite[Theorem 1.2]{DLLW}, for any $1<p<\infty$, we have
\begin{align}\label{l1}
\nonumber L_1
&\leq\left({1\over |B|}\int_B\big| [b,R_j](f_1)(g)\big|^pdg\right)^{1\over p}\\
&\leq C|B|^{-{1\over p}}\|b\|_{{\rm BMO}(\mathcal G)}\left(\int_{r_oB}|f_1(g)|^pdg\right)^{1\over p}\\
\nonumber&\leq C(r_o)\|b\|_{{\rm BMO}(\mathcal G)}\|f\|_{L^\infty(\mathcal G)}.
\end{align}

For any $\tilde g\in B$, we have
\begin{align*}
L_2(\tilde g)\leq {1\over |B|}\int_B \big|b(g)-b_B\big|\big|R_j(f_2)(g)-R_j(f_2)(\tilde g)\big|dg.
\end{align*}
By \eqref{MEHT}, we can obtain
\begin{align*}
\big|R_j(f_2)(g)-R_j(f_2)(\tilde g)\big|
&\leq\|f\|_{L^\infty(\mathcal G)}\sum_{l=\gamma+1}^{\infty}\int_{2^lB\setminus 2^{l-1}B}
\left|K_j(g,g')-K_j(\tilde g, g')\right|dg'\\
&\leq C\|f\|_{L^\infty(\mathcal G)}\sum_{l=\gamma+1}^{\infty}2^{-l}{1\over |2^lB|}\int_{2^lB}dg'\\
&\leq C\|f\|_{L^\infty(\mathcal G)}.
\end{align*}
Therefore, for any $\tilde g\in B$,
\begin{align}\label{l2}
L_2(\tilde g)\leq C\|f\|_{L^\infty(\mathcal G)}
{1\over |B|}\int_B \big|b(g)-b_B\big|dg\leq C\|b\|_{{\rm BMO}(\mathcal G)}\|f\|_{L^\infty(\mathcal G)}.
\end{align}

For the term $L_3$, we use \eqref{MEHT} again, then for any $g,g'\in B$, we have
\begin{align*}
\big|J_3(g',g)\big|
&\leq C\|f\|_{L^\infty(\mathcal G)}\sum_{l=\gamma+1}^{\infty}2^{-l}{1\over |2^lB|}\int_{2^lB}\left|b(g_1)-b_B\right|dg_1\\
&\leq C\|b\|_{{\rm BMO}(\mathcal G)}\|f\|_{L^\infty(\mathcal G)}.
\end{align*}
Thus we have
\begin{align}\label{l3}
L_3\leq C\|b\|_{{\rm BMO}(\mathcal G)}\|f\|_{L^\infty(\mathcal G)}.
\end{align}

From \eqref{deco0}, we have 
\begin{align*}
 \left|J_4(g,\tilde g)\right| &\leq \left|[b,R_j](f)(g)-\big([b,R_j]f\big)_B\right|+|J_1(g)|+\left|\big(J_1\big)_B\right|+\left|J_2(g,\tilde g)\right|\\
&\quad+\left|\big(J_2(\cdot,\tilde g)\big)_B\right|+\left|\left(J_3(\cdot, g)\right)_B\right|,
 \end{align*}
which means
 \begin{align}\label{deco2}
L_4(\tilde g)\leq {1\over |B|}\int_B\big|[b,R_j](f)(g)-\big([b,R_j]f\big)_B\big|dg+2L_1+2L_2(\tilde g)+L_3.
 \end{align}

By \eqref{deco1}, \eqref{deco2} and \eqref{l1}-\eqref{l3}, we can see that $[b,R_j]$ is bounded from $L_c^\infty(\mathcal G)$ to ${\rm BMO}(\mathcal G)$ if and only if $L_4(\tilde g)\leq C\|f\|_{L^{\infty}(\mathcal G)}$ for any $\tilde g\in B$, i.e.,
\begin{align*}
C\|f\|_{L^{\infty}(\mathcal G)}&\geq{1\over |B|}\int_B\left|J_4(g,\tilde g)\right|dg\\
&={1\over |B|}\int_B\left|\big(b(g)-b_B\big)R_j(f_2)(\tilde g)\right|dg\\
&=\left({1\over |B|}\int_{B}\big|b(g)-b_B\big|dg\right)\left|\int_{(r_oB)^c}K_j\big (\tilde g, g'\big)f\big(g'\big)dg'\right|.
\end{align*}
This proves the proposition.
\end{proof}

}

\begin{proof}[Proof of Theorem \ref{thm2}]
From Proposition \ref{prop1} and Proposition \ref{prop2}, we can see that it is suffices to show both \eqref{h1b} and \eqref{lb} are equivalent to the condition (iii).

For the equivalence of \eqref{h1b} and (iii). If $b$ equals  a constant almost everywhere, then
 for any atom $a$ supported in some ball $B=B(g_0,r)$ and for any $\tilde g\in B$,
\begin{align*}
\left(\int_{(r_oB)^c}\left| K_j(g,\tilde g)\right|dg\right)\left|\int_{B}b\big(g'\big)a\big(g'\big)dg'\right|
\leq C\left(\int_{(r_oB)^c}\left| K_j(g,\tilde g)\right|dg\right)\left|\int_{B}a\big(g'\big)dg'\right|=0,
\end{align*}
due to the mean value zero property of atom.

Conversely, assume that \eqref{h1b} holds.  Let $G$ be the set in Corollary \ref{cor1}, then $\inf\limits_{g'\in G}\rho(g_0,g')=r_or $, $|G|=\infty$ and  for $g\in G$, $\tilde g\in B$, we have
$$ |K_j(g,\tilde g)|\geq  C \rho(g,g_0)^{-Q}.$$
Therefore
\begin{align*}
C&\geq \left(\int_{(r_oB)^c}\left| K_j(g,\tilde g)\right|dg\right)\left|\int_{B}b\big(g'\big)a\big(g'\big)dg'\right|dg\\
&\geq  C\left(\int_{ G}\rho(g,g_0)^{-Q}dg\right)\left|\int_{B}b\big(g'\big)a\big(g'\big)dg'\right|,
\end{align*}
where $a$ is any $(1,p)$-atom supported in $B$ with $1<p<\infty$. This  is impossible unless 
\begin{align}\label{ba=0}
\int_{B}b(g')a(g')dg'=0
\end{align} 
for every ball $B$ and any $(1,p)$-atom supported in $B$ with $1<p<\infty$. 
We recall a result from \cite{MSV}: one can define the space $H^1_{\rm fin}(\mathcal G)$ as the set of
all finite linear combinations of $(1,p)$-atoms, which is endowed with the natural norm
$$ \|f\|_{ H^1_{\rm fin}(\mathcal G) }=\inf\Big\{ \Big( \sum_{j=1}^N |\lambda_j|^p \Big)^{1\over p}:\ f=\sum_{j=1}^N \lambda_ja_j, \  a_j  \ (1,p)-{\rm atoms}, N\in\mathbb N \Big\}.$$  
Note that $H^1_{\rm fin}(\mathcal G)$ is dense in $H^1(\mathcal G)$. Moreover, from \cite[Proposition 2]{MSV}, we know that the two norms $\|\cdot \|_{H^1_{\rm fin}(\mathcal G)}$ and $\|\cdot \|_{H^1(\mathcal G)}$ are equivalent on $H^1_{\rm fin}(\mathcal G)$. Hence, if \eqref{ba=0} holds for every $(1,p)$-atom, then we obtain that
$b$ is a zero linear functional on $H^1_{\rm fin}(\mathcal G)$, and hence extends to a zero linear functional on $H^1(\mathcal G)$. This shows that $b$ is in ${\rm BMO}(\mathcal G)$ with 
$$ \|b\|_{{\rm BMO}(\mathcal G)}=0. $$
Thus, $b$ equals  a constant almost everywhere.

For the equivalence of \eqref{lb} and (iii). It is easy to see that if $b$ equals a constant almost everywhere, then \eqref{lb} holds. Conversely, take $f_N(g)=\chi_{G\cap B(g_0,N)}(g)$, $N\in\mathbb N$, in \eqref{lb} to obtain
\begin{align*}
C&\geq\left({1\over |B|}\int_{B}\big|b(g)-b_B\big|dg\right)\left|\int_{G\cap B(g_0,N)}K_j\big (\tilde g, g'\big)dg'\right|\\
&=\left({1\over |B|}\int_{B}\big|b(g)-b_B\big|dg\right)\int_{G\cap B(g_0,N)}\left|K_j\big (\tilde g, g'\big)\right|dg'\\
&\geq C_1\left({1\over |B|}\int_{B}\big|b(g)-b_B\big|dg\right)\int_{G\cap B(g_0,N)}\rho(g',g_0)^{-Q}dg'\\
&=C_2 \log N\left({1\over |B|}\int_{B}\big|b(g)-b_B\big|dg\right),
\end{align*}
for all $N\in\mathbb N$ large enough. Letting $N$ go to infinity we have $b(g)=b_B$ a.e. in $B$, and hence $ b$ must be constant almost everywhere.
\end{proof}

\section{Endpoint characterisation of commutator $[b,R_j]$ on Heisenberg groups $\mathbb H^n$ and proof of Theorem \ref{thm3}}
\setcounter{equation}{0}

\begin{proof}[Proof of Theorem \ref{thm H lower bound}]
{We handle the Riesz transform kernel by using the idea in the proof of \cite[Proposition 3.1]{LHQ}.

{Recall that (see for example \cite{Hu} and \cite{Ga}) the explicit expression of heat kernel on the Heisenberg group $\mathbb H^n$ is
as follows: for $g=(z,t)\in\H^n$,
$$p_h(g)={1\over 2(4\pi h)^{n+1}}\int_{\mathbb R}\exp\Big({\lambda\over 4h}\big(it-\|z\|^2\coth\lambda\big)\Big)\Big({\lambda\over\sinh\lambda}\Big)^nd\lambda,$$
where  $\|z\|=\sum_{j=1}^n\|z_j\|^2$.}

For any $g=(z,t)\in\mathbb H^n$},  by using the explicit expression of the heat kernel above and by Fubini's theorem, we have
\begin{align*}
(-\Delta_{\mathbb H^n})^{-{1\over 2}}(g)
&=C\int_{0}^{+\infty}h^{-{1\over 2}}p_h(g)dh\\
&=C'\int_{\mathbb R}\,\int_{0}^{+\infty}h^{-n-{3\over 2}}\exp\Big({\lambda\over 4h}\big(it-\|z\|^2\coth\lambda\big)\Big)dh\ \Big({\lambda\over\sinh\lambda}\Big)^nd\lambda\\
&=C^{''}\int_{\mathbb R}\big(\|z\|^2\lambda\coth\lambda-i\lambda t\big)^{-n-{1\over 2}}\Big({\lambda\over\sinh\lambda}\Big)^nd\lambda.
\end{align*}
Then by \eqref{XYT}, for $j=1,\cdots, n$, we can obtain
\begin{align*}
{X_j (-\Delta_{\mathbb H^n})^{-\frac{1}{2}}(g)}&=C\Big({\partial\over\partial x_j}+2y_j{\partial\over\partial t}\Big)\int_{\mathbb R}\big(\|z\|^2\lambda\coth\lambda-i\lambda t\big)^{-n-{1\over 2}}\Big({\lambda\over\sinh\lambda}\Big)^nd\lambda\\
&=C(-2n-1)\bigg[x_j\int_{\mathbb R}\big(\|z\|^2\lambda\coth\lambda-i\lambda t\big)^{-n-{3\over 2}}\Big({\lambda\over\sinh\lambda}\Big)^{n+1}\cosh\lambda d\lambda\\
&\hskip4cm -iy_j\int_{\mathbb R}\big(\|z\|^2\lambda\coth\lambda-i\lambda t\big)^{-n-{3\over 2}}\Big({\lambda\over\sinh\lambda}\Big)^{n}\lambda d\lambda\bigg].
\end{align*}

Observe that
\begin{align*}
\|z\|^2\lambda\coth\lambda-i\lambda t 
&={\lambda\over \sinh\lambda}d_K^2(g)\bigg({\|z\|^2\over d_K^2(g)}\cosh\lambda-i{t\over d_K^2(g)}\sinh\lambda\bigg)\\
&={\lambda\over \sinh\lambda}d_K^2(g)\cosh(\lambda-i\phi),
\end{align*}
where
$$-{\pi\over2}\leq\phi=\phi(\|z\|,t)\leq {\pi\over 2},\quad e^{i\phi}=d_K^{-2}(g)(\|z\|^2+i\,t),$$
and $d_K(g)$ is the Kor\'anyi norm as defined in \eqref{dk}.
Therefore,
\begin{align*}
{X_j (-\Delta_{\mathbb H^n})^{-\frac{1}{2}}(g)}&
=Cd_K^{-Q-1}(g)\bigg[x_j\int_{\mathbb R}\Big({\lambda\over\sinh\lambda}\Big)^{-{1\over 2}}\cosh\lambda \big(\cosh(\lambda-i\phi)\big)^{-n-{3\over 2}}d\lambda\\
&\hskip4cm -iy_j\int_{\mathbb R}\Big({\lambda\over\sinh\lambda}\Big)^{-{3\over 2}}\lambda \big(\cosh(\lambda-i\phi)\big)^{-n-{3\over 2}}d\lambda\bigg].
\end{align*}
Then by the Cauchy integral theorem, we have
\begin{align*}
{X_j (-\Delta_{\mathbb H^n})^{-\frac{1}{2}}(g)}=Cd_K^{-Q-1}(g)F_j(g),
\end{align*}
where
\begin{align*}
F_j(g)&=x_j\int_{\mathbb R}\Big[{\sinh(\lambda+i\phi)\over \lambda+i\phi}\Big]^{1\over 2}\cosh(\lambda+i\phi)(\cosh \lambda)^{-n-{3\over 2}}d\lambda\\
&\quad-iy_j\int_{\mathbb R}(\lambda+i\phi)\Big[{\sinh(\lambda+i\phi)\over \lambda+i\phi}\Big]^{3\over 2}(\cosh \lambda)^{-n-{3\over 2}}d\lambda,
\end{align*}

Similarly,
\begin{align*}
{X_{n+j}(-\Delta_{\mathbb H^n})^{-\frac{1}{2}}(g)}=Cd_K^{-Q-1}(g)H_j(g),
\end{align*}
where
\begin{align*}
H_j(g)&=y_j\int_{\mathbb R}\Big[{\sinh(\lambda+i\phi)\over \lambda+i\phi}\Big]^{1\over 2}\cosh(\lambda+i\phi)(\cosh \lambda)^{-n-{3\over 2}}d\lambda\\
&\quad+ix_j\int_{\mathbb R}(\lambda+i\phi)\Big[{\sinh(\lambda+i\phi)\over \lambda+i\phi}\Big]^{3\over 2}(\cosh \lambda)^{-n-{3\over 2}}d\lambda.
\end{align*}

Let
\begin{align*}
A_n(w)&=\int_{\mathbb R}\Big[{\sinh(\lambda+w)\over \lambda+w}\Big]^{1\over 2}\cosh(\lambda+w)(\cosh \lambda)^{-n-{3\over 2}}d\lambda,\quad w\in\mathbb C,\\
B_n(w)&=\int_{\mathbb R}(\lambda+w)\Big[{\sinh(\lambda+w)\over \lambda+w}\Big]^{3\over 2}(\cosh \lambda)^{-n-{3\over 2}}d\lambda\quad w\in\mathbb C.
\end{align*}
Then 
\begin{align*}
F_j(g)=x_jA_n(i\phi)-iy_jB_n(i\phi),\quad H_j(g)=y_jA_n(i\phi)+ix_jB_n(i\phi).
\end{align*}

{Notice that} $A_n(w)$ and $B_n(w)$ are analytic in some domain on $\mathbb C$, which contains the segment $[-{\pi i\over 2}, {\pi i\over 2}]$ of the imaginary axis, and $A_n(0)\neq 0$, $B_n(0)=0$. Thus, 
$A_n(i\phi)$ has at most a finite number of zero points on $[-{\pi \over 2}, {\pi \over 2}]$, i.e., there exist $\{\phi_\ell\}_{\ell=1}^N\subset [-{\pi \over 2}, {\pi \over 2}]$ such that $A_n(i\phi_\ell)=0$. By noting that $\phi=\phi(\|z\|^2,t)$, we see that
$\{\phi_\ell\}_{\ell=1}^N$ corresponds to a set $\mathcal H_N$ in $\H^n$ with
$$ \mathcal H_N:=\{ (z,t)\in\H^n: \phi_\ell=\phi(\|z\|^2,t), \ell=1,\ldots,N\}, $$
 which has measure zero.

Therefore,  for any fixed $\phi \in [-{\pi \over 2}, {\pi \over 2}]\backslash \{\phi_\ell\}_{\ell=1}^N$, when we fixed $|z_j|^2=x_j^2+y_j^2$ with $x_j\cdot y_j\not=0$, there are at most two $z_j$ satisfying $F_j(g)=0$ (or $H_j(g)=0$). Consequently, the measure of the set of $g$ satisfying $F_j(g)=0$ (or $H_j(g)=0$) is zero.  This completes the proof.
\end{proof}

\begin{proof}[Proof of Theorem \ref{thm3}]
We first prove  the sufficient part. Suppose $j\in\{1,\ldots,2n\}$ and $b\in L^\infty(\mathbb H^n)$ with $\|b\|_{L^\infty(\mathbb H^n)}\not=0$.
For $f\in L^1(\mathbb H^n)$, and for any $\lambda>0$, we have
\begin{align*}
&|\{ g\in\mathbb H^n: |[b,R_j](f)(g)|>\lambda \}|\\
&\leq|\{ g\in\mathbb H^n: \big|b(g)R_j(f)(g) \big|>\lambda/2 \}|+|\{ g\in\mathbb H^n: \big|  R_j(bf)(g) \big|>\lambda/2 \}|\\
&\leq  C\|b\|_{L^\infty(\mathbb H^n)}{   \|f\|_{L^1(\mathbb H^n)} \over\lambda },
\end{align*}
which shows that $[b,R_j]$ is of weak type $(1,1)$, where the last inequality follows from the fact that
$R_j$ is of weak type $(1,1)$.

For the necessity part. Suppose that $b\in L^1_{loc}(\mathbb H^n)$, then $b$ is finite almost everywhere and almost every point is a Lebesgue point of $b$. 

Let $f={1\over |B(0,1)|}\chi_{B(0,1)}$.
For  every $\epsilon>0$, set $f_\epsilon (g)={1\over \epsilon^Q}f(\delta_{\epsilon^{-1}}(g))$ and $f_\epsilon ^{g'}(g)=f_\epsilon (g'^{-1}\circ g)$. Fix any Lebesgue point $g'$ of $b$, since $K_j\in C^\infty(\mathbb H^n\setminus \{0\})$, for any $g\neq g'$, we have
\begin{align*}
\lim_{\epsilon\rightarrow 0}\left| [b,R_j] \big(f_\epsilon ^{g'}\big)(g) \right|
&=\lim_{\epsilon\rightarrow 0}\left| {\rm p.v.} \int_{\mathbb H^n} K_j(g,\tilde g) \big(b(g)-b(\tilde g)\big)f_\epsilon (g'^{-1}\circ \tilde g)d\tilde g \right|\\
&=\lim_{\epsilon\rightarrow 0}{1\over |B(g',\epsilon)|} \left| {\rm p.v.} \int_{B(g',\epsilon)} K_j(g,\tilde g) \big(b(g)-b(\tilde g)\big)d\tilde g \right|\\
&=\left|K_j(g, g') \right|    \left|b(g)-b( g')\right|.
\end{align*}

\color{black}
Thus,
\begin{align}\label{kb est}
\left|\left\{g\in\mathbb H^n\setminus\{g'\} :\left|K_j(g, g') \right|    \left|b(g)-b( g')\right|>\lambda\right\}\right|
\leq {\|[b,R_j]\|_{L^1(\mathbb H^n)\rightarrow L^{1,\infty}(\mathbb H^n)}\over {\lambda}}.
\end{align}

By Theorem \ref{thm H lower bound}, we can see that $K_j(g)\neq 0$ almost everywhere on $S(0,1)$.
Fix small $\varepsilon>0$ and take $\Gamma$ to be a compact subset of $S(0,1)$ such that $K_j(g)\neq 0$ on $\Gamma$ and $\sigma(S(0,1)\setminus \Gamma)<\varepsilon$, where $\sigma$ is the Radon measure on $S(0,1)$. Let $C_{K_j}=\inf\{|K_j(g)|: g\in \Gamma\}$, since $K_j\in C^\infty(\mathbb H^n\setminus\{0\})$, $j=1,\cdots, 2n$, we have $C_{K_j}>0$.

Set
\begin{align*}
S_\Gamma(g')&=\left\{g\in\mathbb H^n: \delta_{d_K(g,g')^{-1}}(g'^{-1}\circ g)\in \Gamma\right\},\\
\Lambda_{\lambda}(g')&=\bigg\{g\in\mathbb H^n: \delta_{d_K(g,g')^{-1}}(g'^{-1}\circ g)\in \Gamma, {|b(g)-b(g')|\over d_K(g,g')^Q}>\lambda\bigg\}.
\end{align*}
Then for any $r>0$, we have
\begin{align}\label{gamma mea}
 \big|B(0,r)\setminus S_\Gamma(0)\big|<\varepsilon {r^Q\over Q}.
\end{align}

By the homogeneous property of $K_j$ \eqref{kjs} and \eqref{kb est}, we have
\begin{align*}
\left|\Lambda_\lambda(g')\right|
&\leq\bigg|\bigg\{g\in\mathbb H^n: \delta_{d_K(g,g')^{-1}}(g'^{-1}\circ g)\in \Gamma, |b(g)-b(g')| |K_j(g,g')|>C_{K_j}\lambda\bigg\}\bigg|\\
&\leq{1\over {C_{K_j}\lambda}}\big\|[b,R_j]\big\|_{L^1(\mathbb H^n)\rightarrow L^{1,\infty}(\mathbb H^n)}.
\end{align*}

Since  $[b-c, R_j]=[b, R_j]$, $j=1,\cdots, 2n$,  for any constant $c$ and $b$ is finite almost everywhere, we may assume $b(0)=0$, then we have
\begin{align}\label{lam0}
\left|\Lambda_\lambda(0)\right|
&=\left|\left\{g\in\mathbb H^n: \delta_{d_K(g)^{-1}}( g)\in \Gamma, {|b(g)|\over d_K(g)^Q}>\lambda\right\}\right|\\
\nonumber&\leq{1\over {C_{K_j}\lambda}}\|[b,R_j]\|_{L^1(\mathbb H^n)\rightarrow L^{1,\infty}(\mathbb H^n)}.
\end{align}

Let $g'\neq 0$, $g\in B(g', {1\over 2}d_K(g')|b(g')|^{1/Q})\cap S_\Gamma(g')$ and $g\notin\Lambda_{d_K(g')^{-Q}}(g')$, then
\begin{align*}
|b(g)|&\geq \left|b(g')\right|-{|b(g)-b(g')|\over d_K(g,g')^Q}d_K(g,g')^Q
\geq \Big(1-{1\over 2^Q}\Big)\left|b(g')\right|
\end{align*}
for almost every $g'\in\mathbb H^n$.
 Then by \eqref{lam0}, we have
\begin{align}\label{less}
\nonumber I_{g',\Gamma}:&=\bigg|\bigg\{g\in B\Big(g', {1\over 2}d_K(g')|b(g')|^{1\over Q}\Big)\cap S_\Gamma(g')\cap S_\Gamma(0)\setminus\Lambda_{d_K(g')^{-Q}}(g'): \\
\nonumber&\qquad {|b(g')|\over d_K(g)^Q}>{1\over C_{d_K}^Qd_K(g')^Q}\bigg\}\bigg|\\
&\leq\bigg|\bigg\{g\in S_\Gamma(0):  {|b(g)|\over d_K(g)^Q}>{1-2^{-Q} \over C_{d_K}^Qd_K(g')^Q}\bigg\}\bigg|\\
\nonumber&\leq{C_{d_K}^Qd_K(g')^Q\over {C_{K_j}(1-2^{-Q})}}\big\|[b,R_j]\big\|_{L^1(\mathbb H^n)\rightarrow L^{1,\infty}(\mathbb H^n)},
\end{align}
where $C_{d_K}$ is the constant in \eqref{qdr1}.

Suppose that $|b(g')|\geq 2^Q$, then for any $g\in B(g', {1\over 2}d_K(g')\left|b(g')\right|^{1\over Q})$, by \eqref{qdr}, 
\begin{align*}
d_K(g)\leq C_{d_K}\!\!\left(d_K(g,g')+d_K(g',0)\right)\leq C_{d_K} \Big({1\over 2}d_K(g')\left|b(g')\right|^{1\over Q}+d_K(g')\Big)
\leq  C_{d_K}d_K(g')\left|b(g')\right|^{1\over Q}.
\end{align*}
That is,
$$B\big(g', {1\over 2}d_K(g')\left|b(g')\right|^{1\over Q}\big)\subset B\big(0, C_{d_K}d_K(g')\left|b(g')\right|^{1\over Q}\big).$$
Therefore,
\begin{align*}
 I_{g',\Gamma}\!&\geq\!\Big|B\Big(g', {1\over 2}d_K(g')|b(g')|^{1\over Q}\Big)\cap S_\Gamma(g')\Big|\!-\!\Big|
B\big(0, C_{d_K}d_K(g')\left|b(g')\right|^{1\over Q}\big)\!\setminus\! S_\Gamma(0)\Big|\!-\!\big|\Lambda_{d_K(g')^{-Q}}(g')\big|.
\end{align*}

Observe that, by \eqref{gamma mea}, we have
\begin{align*}
&\Big|B\Big(g', {1\over 2}d_K(g')|b(g')|^{1\over Q}\Big)\cap S_\Gamma(g')\Big|\\
&=\Big|B\Big(0, {1\over 2}d_K(g')|b(g')|^{1\over Q}\Big)\cap S_\Gamma(0)\Big|\\
&=\Big|B\Big(0, {1\over 2}d_K(g')|b(g')|^{1\over Q}\Big)\Big|- \Big|B\Big(0, {1\over 2}d_K(g')|b(g')|^{1\over Q}\Big)\cap S_\Gamma(0)^{c}\Big|\\
&> {1\over Q 2^Q}\big(\omega_Q-\varepsilon\big)d_K(g')^Q|b(g')|,
\end{align*}
where $\omega_Q$ is the Radon measure of $S(0,1)$,
and
\begin{align*}
&\Big|
B\big(0, C_{d_K}d_K(g')\left|b(g')\right|^{1\over Q}\big)\setminus S_\Gamma(0)\Big|+\left|\Lambda_{d_K(g')^{-Q}}(g')\right|\\
&
<{\varepsilon\over Q}C_{d_K}^Qd_K(g')^Q|b(g')|+{1\over {C_{K_j}}}d_K(g')^Q \|[b,R_j]\|_{L^1(\mathbb H^n)\rightarrow L^{1,\infty}(\mathbb H^n)}.
\end{align*}
Consequently, we obtain
$$ I_{g',\Gamma}>{1\over Q}d_K(g')^Q|b(g')|\Big({\omega_Q\over 2^Q}-{1+C_{d_K}^Q2^Q\over 2^Q}\varepsilon\Big)-{1\over {C_{K_j}}}d_K(g')^Q\|[b,R_j]\|_{L^1(\mathbb H^n)\rightarrow L^{1,\infty}(\mathbb H^n)}.$$

Now take $\varepsilon={\omega_Q\over 2(1+C_{d_K}^Q2^Q)}$, we have
\begin{align}\label{more}
 I_{g',\Gamma}>{\omega_Q\over 2^{Q+1}Q}d_K(g')^Q|b(g')|-{1\over {C_{K_j}}}d_K(g')^Q\|[b,R_j]\|_{L^1(\mathbb H^n)\rightarrow L^{1,\infty}(\mathbb H^n)}.
\end{align}

Now
combining the inequalities \eqref{less} and \eqref{more}, we obtain that
$$\left|b(g')\right|<{2^{Q+1}Q\over C_{K_j}\omega_Q}\Big(1+{2^QC_{d_K}^Q\over 2^Q-1}\Big)
\|[b,R_j]\|_{L^1(\mathbb H^n)\rightarrow L^{1,\infty}(\mathbb H^n)}.$$

To sum up, for almost all $g'\in\mathbb H^n$,
$$\left|b(g')\right|\leq\max\bigg\{2^Q, {2^{Q+1}Q\over C_{K_j}\omega_Q}\Big(1+{2^QC_{d_K}^Q\over 2^Q-1}\Big)\|[b,R_j]\|_{L^1(\mathbb H^n)\rightarrow L^{1,\infty}(\mathbb H^n)}\bigg\}.$$
This completes the proof of Theorem \ref{thm3}.
\end{proof}

\bigskip
{\bf Acknowledgement:} X. T. Duong and J. Li are supported by ARC DP 160100153. H.-Q. Li
is partially supported by NSF of China (Grants No. 11625102 and No. 11571077) and ``The
Program for Professor of Special Appointment (Eastern Scholar) at Shanghai Institutions
of Higher Learning''.
 B. D. Wick's research is partially supported by National Science Foundation -- DMS \# 1560955.
 Q. Y. Wu is supported by NSF of China (Grants No. 11671185 and No. 11701250) and the State Scholarship Fund of China (No. 201708370017).

\bibliographystyle{amsplain}

\begin{thebibliography}{10}



%
%




\bibitem{Ac}  N. Accomazzo,  A characterization of BMO in terms of endpoint bounds for commutators of singular integrals, arXiv:1802.05516.



\bibitem{B} S. Bloom, A commutator theorem and weighted BMO,  \textit{Trans. Amer. Math. Soc.},
  {\bf292} (1985),
  103--122.




\bibitem{BLU} A. Bonfiglioli, E. Lanconelli, F. Uguzzoni, Stratied Lie Groups and Potential Theory for their Sub-Laplacians. Springer Monographs in Mathematics 2007.

\bibitem{C}A.P. Calder\'on, Commutators of singular integral operator,  \textit{Proc. Nat. Acad. Sci. USA}, {\bf53} (1965), 1092--1099.


\bibitem{Chr} {M. Christ,} A $T\left( b\right) $ theorem with remarks
on analytic capacity and the Cauchy integral, \textit{Colloq. Math.}, \textbf{%
61} (1990), 601--628.

\bibitem{CG} {M. Christ and D. Geller,} Singular integral characterisations of Hardy spaces on homogeneous groups, \textit{Duke. Math.}, \textbf{%
61} (1990), 601--628.



\bibitem{CRW}
{ R. Coifman, R. Rochberg and G. Weiss},
   Factorization theorems for Hardy spaces in several variables,
   \textit{Ann. of Math. (2)},
   \textbf{103}
   (1976), 611--635.


%
%
%
%

\bibitem{DLHWY} X.T. Duong, I. Holmes, J. Li, B.D. Wick and D. Yang, Two weight Commutators in the Dirichlet and Neumann Laplacian settings,
arXiv:1705.06858.


\bibitem{DLLW}  {X. T. Duong, H.-Q. Li, J. Li and B. D. Wick, Lower bound for Riesz transform kernels and commutator theorems on stratified nilpotent Lie groups}, to appear in J. Math. Pure. Appl.

\bibitem{DLWY}
X. T. Duong, J. Li, B. D. Wick and D. Y. Yang,  Factorization for Hardy spaces and characterization for BMO spaces via commutators in the Bessel setting, \textit{Indiana Univ. Math. J.}, {\bf66} (2017), no. 4, 1081--1106.


%
%

\bibitem {FoSt} G.B. Folland and E.M. Stein, Hardy Spaces on Homogeneous Groups, Princetion University Press,
Princeton, N.J., 1982.







\bibitem{Ga}
B. Gaveau, Principe de moindre action, propagation de la chaleur et estim\'ees sous elliptiques sur certains groupes nilpotents,  \textit{Acta Math.}, {\bf139} (1977), 95--153.

\bibitem{GLW} W.C. Guo, J.L. Lian and H.X. Wu,
The unified theory for the necessity of bounded commutators and applications,
arXiv:1709.08279.


\bibitem{HST} E. Harboure, C. Segovia and J. Torrea, 
Boundedness of commutators of fractional and singular integrals for the extreme values of $p$,
\textit{Illinois J. Math.}, {\bf 41} (1997),  676--700. 


\bibitem{HLW}
   I. Holmes, M. Lacey, B. D. Wick,
   Commutators in the two-weight setting,
   \textit{Math. Ann.},
   {\bf 367} (2017),
   51--80.


\bibitem{HK}      T.~Hyt\"onen and A.~Kairema,
    Systems of dyadic cubes in a doubling metric space,
    \emph{Colloq. Math.}, \textbf{126} (2012), no.~1, 1--33.
    
    
\bibitem{Hu} A. Hulanicki, 
The distribution of energy in the Brownian motion in the Gaussian field and analytic-hypoellipticity of certain subelliptic operators on the Heisenberg group,
 \textit{Studia Math.}, {\bf56} (1976), 165--173.




\bibitem{KLPW} A. Kairema,  J. Li, C. Pereyra and L.A. Ward, Haar bases on quasi-metric measure spaces, and dyadic structure theorems for function spaces on product spaces of homogeneous type,
 \textit{J.  Funct. Anal.}, {\bf271} (2016), 1793--1843.


\bibitem{Ler} A.K. Lerner, S. Ombrosi, I.P. Rivera-R\'ios, Commutators of singular integrals revisited, arXiv:1709.04724.

\bibitem{LHQ} H. -Q. Li, 
Fonctions maximales centr\'ees de Hardy-Littlewood sur les groupes de Heisenberg,
\textit{Studia Math.}, {\bf 191} (2009),  89--100. 

\bibitem{LNWW} J. Li, T. Nguyen, L.A. Ward and B.D. Wick, The Cauchy integral, bounded and compact commutators,
arXiv:1709.00703.

\bibitem{LW} J. Li and B. D. Wick, Characterizations of $H^1_{\Delta_N}(\mathbb R^n)$ and ${\rm BMO} _{\Delta_N}(\mathbb R^n)$ via Weak Factorizations and Commutators,
\textit{J. Funct. Anal.}, {\bf 272} (2017), 5384--5416.

\bibitem{MSV} S. Meda, P. Sj\"ogren and M. Vallarino, Atomic decompositions and operators on Hardy sapces,
 \textit{Revista de La Uni\'on Matem\'atica Argentina}, {\bf50} (2009), 15--22.







%
%
%

\bibitem{P} C. P\'erez, Endpoint estimates for commutators of singular integral operators,
\textit{J. Funct. Anal.}, {\bf128} (1995), 163--185. 


\bibitem{TYY} J. Tao, Da. Yang and Do. Yang, Boundedness and compactness characterizations of Cauchy integral commutators on Morrey spaces,  arXiv:1801.04997.


\bibitem {To} A. Torchinsky, Real-variable methods in harmonic analysis, Pure and Applied Mathematics, vol. 123, Academic Press, Inc., Orlando, FL, 1986.



\bibitem{U}
A. Uchiyama, On the compactness of operators of Hankel type,  \textit{T\^ohoku Math. J.}, (2) {\bf30} (1978), no. 1, 163--171.




\bibitem{VSC92} N. Th. Varopoulos, L. Saloff-Coste and T. Coulhon, Analysis and
geometry on groups, Cambridge Tracts in Mathematics, 100. Cambridge
University Press, Cambridge, 1992.


\end{thebibliography}

\end{document}